\documentclass{cimart}
\usepackage{longtable,cleveref}
 \usepackage[symbol]{footmisc}
 
\title[Integro-derivation Dzhumadildaev algebras: from the algebra of polynomials]{Integro-derivation Dzhumadildaev algebras: \\from the algebra of polynomials}

\authors{Ivan Kaygorodov and Naurizbay Uzakbaev}

\authorinfo[I.~Kaygorodov]
    {University of  Beira Interior, Portugal}
    {kaygorodov.ivan@gmail.com}
    
\authorinfo[N.~Uzakbaev]
    {Romanovsky Institute of Mathematics and Karakalpak State University, Uzbekistan}
    {nn.uzakbaev@gmail.com}

\abstract{This paper introduces and investigates some properties of algebras constructed from the algebra of polynomials via derivation and integration operators using a process presented by Dzhumadildaev in an earlier work. In particular, we discover new classes of infinite-dimensional simple conservative algebras and describe the derivations of these algebras for ranks $1$ and $2$.}



\VOLUME{34}
\YEAR{2026}
\ISSUE{1}
\NUMBER{15}
\DOI{https://doi.org/10.46298/cm.17612}
\editinfo{March 2, 2026}{March 6, 2026}{Pasha Zusmanovich}
\acknowledgments{The work is supported by FCT 2023.08031.CEECIND and UID/00212/2025.}

\begin{document}

\section{Introduction}

The idea of obtaining new objects from old ones by using derivative operations has long been known in algebra \cite{albert}. In its most general form, the idea was realized by  Malcev \cite{malcev}. Let $\rm M_n$ be the algebra of matrices of order $n$ over a field. Assume that some finite collection $\Lambda=\big(a_{ij} ,b_{ij} ,c_{ij} \big)$ of matrices in $\rm M_n$ is given. Denote by $\rm M^{(\Lambda)}_n$ an algebra defined on a space of matrices in $\rm M_n$ with respect to new multiplication \begin{center} $x \cdot _{\Lambda} y = \sum\limits_{i,j} a_{ij}xb_{ij}yc_{ij}.$ \end{center}
It was proved that every $n$-dimensional algebra  is isomorphic to a subalgebra of $\rm M^{(\Lambda)}_n$ \cite{malcev}. Other interesting ways to derive the initial multiplication are isotopes, homotopes, and mutations \cite{alsaody,tk24,pche2}. The concept of an isotope was introduced by Albert \cite{albert}. Let algebras $\rm A$ and $\rm A_0$ have a common linear space on which right multiplication operators $R_x$ and $R^{(0)}_x$ are defined (for $\rm A$ and $\rm A_0$, resp.). We say that $\rm A$ and $\rm A_0$ are isotopic if there exist invertible linear operators $\phi, \psi, \xi$ such that $R^{(0)}_x = \phi R_{x\psi} \xi.$ We call $\rm A_0$ an isotope of $\rm A$. Another similar construction appears in the study of Novikov algebras and their generalizations. Namely, let $\rm A$ be an associative commutative algebra with a derivation $d,$ then a new multiplication $x \ast y = xd(y)$ gives a structure of a Novikov algebra \cite{d25}. The present construction was generalized to the case of noncommutative Novikov algebras  \cite{SK} and $\delta$-Novikov algebras \cite{dN}. On the other hand, an associative commutative algebra with a Rota-Baxter operator $R$ gives a structure of a Zinbiel algebra under the multiplication $x*y= xR(y)$ \cite{d25}. Recently, in a paper of Dzhumadildaev \cite{d25}, an idea of mixing the two above-presented constructions was introduced. The present paper is dedicated to the study of non-associative algebras, which we call {\it integro-derivation Dzhumadildaev} algebras $\big($see, Definition \ref{defidd}$\big),$ obtained analogously by the construction of Dzhumadildaev given in \cite{d25}.

The algebra of restricted polynomials (also known as null-filiform associative algebra) is still a subject of interest \cite{DET,to26,KLP,KKL}. The present paper is dedicated to the study of some algebras obtained from the $n$-dimensional algebra of restricted polynomials (and their unital versions) under a specific multiplication given by derivation-integration operators. Namely, we give the definition of algebras under our consideration in Proposition~\ref{defidd} and define the exact table of multiplications of algebras under our consideration in Definition \ref{mtable}. Proposition \ref{iddtrivial} gives a characterization of nontrivial algebras under our consideration, and the next subsection \ref{spn} characterizes simple, perfect, and nilpotent algebras. Subsection \ref{cons} provides some new examples of simple conservative algebras. The second part of the paper is dedicated to the study of derivations of integro-derivation Dzhumadildaev algebras of ranks $1$ and $2.$ The description of algebras of derivations and their generalizations of associative and non-associative algebras is a classical, but still current problem \cite{kh22,JW25,B24,BB25,AEK23,AGL25}.

\medskip 

\noindent {\bf Notations} We do not provide some well-known definitions $\big($such as definitions of Lie algebras, Leibniz algebras, nilpotent algebras, solvable algebras,  etc.$\big)$, and refer the readers to consult previously published papers. For a set of vectors $S,$ we denote by  $\big\langle S \big\rangle$ the vector space generated by $S.$ For an algebra ${\rm A}$ and an element $a\in {\rm A},$ we denote the right multiplication on $a$ as ${\rm R}_a.$ In general, we are working in the complex field, but some results are applicable to other fields as well. We also always  assume algebras  under our consideration are nontrivial,  i.e., they have nonzero multiplications.

\section{Preliminaries}

\begin{definition}\label{defidd}
Let ${\rm K}[x]$ be a polynomial algebra $\big($with a unit $x^0\big),$ generated by $x$, and let $\rm S$ be a set of elements $\{ x^j \}_{j \in J},$ where $J \subseteq \mathbb N \cup \big\{0\big\}.$ Then $(\rm S, \diamond)$ is an integro-derivation Dzhumadildaev algebra of type $(n,m)\footnote{We will denote it as {\rm IDD(\rm S, n, m)}.},$ where $n,m \in \mathbb{Z},$ if 
\begin{center}
$x^i\diamond x^j= \begin{cases}
T^n(x^i)T^m(x^j), \mbox{ if } T^n(x^i)T^m(x^j) \mbox{ is linearly dependent with one }x^l, l \in J;\\
0, \mbox{ otherwise}, 
\end{cases}$

\end{center}
where, for any positive $k>0,$ we define as follows
 
\begin{longtable}{l}
$T^k$ is the $k$th partial derivation $"\partial^k",$ where $\partial(x^i)\ =\   ix^{i-1};$\\
$T^0 = {\rm Id},$ and\\
$T^{-k}$ is the $k$th usual integral $"\int\ldots\int"$, where $\int(x^i) \ = \ \frac{1}{i+1}x^{i+1}.$
\end{longtable}

\noindent Rank of ${\rm IDD}({\rm S}, n, m)$ is equal to $|n|+|m|$ and denoted as ${\rm Rank} \big( {\rm IDD}({\rm S}, n, m)\big).$

\noindent Level  of ${\rm IDD}({\rm S}, n, m)$ is equal to $n+m$ and denoted as ${\rm lev} \big( {\rm IDD}({\rm S}, n, m)\big).$

\end{definition}

\begin{proposition}\label{iso}
${\rm IDD({\rm S},n,m)}$ is opposite to  ${\rm IDD({\rm S}, m,n)}.$ In particular, ${\rm IDD({\rm S}, n,n)}$ is commutative.
\end{proposition}

\begin{proof}
It follows from Definition \ref{defidd} and commutativity of ${\rm K}[x].$
\end{proof}

Below, due to Proposition \ref{iso}, we will consider algebras  ${\rm IDD({\rm S},n,m)}$ only with $n\geq m.$

\begin{definition}\label{mb}
Let $B=\big\{ b_i \big\}_{i \in I}$ be a basis of an algebra ${\rm A}.$ It is said that $B$ is a multiplicative basis\footnote{About algebras with  a multiplication basis see \cite{k24} and references therein.} if  for each $e_i$ and $e_j$ from $B,$ we have $e_ie_j \in \big \langle e_k \big\rangle$ for some $k.$ A multiplicative basis $B=\big\{ b_i \big\}_{i \in I}$  of an algebra ${\rm A}$ is called strong multiplicative if for each $e_i$ and $e_j$ from $B,$ we have $e_ie_j\neq0.$
\end{definition}
 
\begin{proposition}
${\rm IDD({\rm S},n,m)}$ has a multiplicative basis.
\end{proposition}

\begin{proof}
It follows from Definitions \ref{defidd} and \ref{mb}.
\end{proof}

\subsection{The multiplication table of \texorpdfstring{${\rm IDD}$}{IDD} algebras from polynomials}

Below, we will consider  IDD algebras, constructed from  the following sets 
 
\begin{longtable}{lcl}
$K_n^k $&$ := $&$ \big\{ {x^k}=e_{k}, \ {x^{k+1}}=e_{k+1},\ \ldots,\ {x^{n}}=e_{n}\big\}$; \\
$K^k_\infty $&$ := $&$ \big\{ {x^k}=e_{k}, \ {x^{k+1}}=e_{k+1},\ \ldots \ \big\}.$ \end{longtable}

Let us remember that ${\rm IDD}(K^1_n, m_1,m_2)$ \ $\big($resp., ${\rm IDD}(K^0_n, m_1,m_2) \big)$ algebras of rank $0$ are ${\rm IDD}(K^1_n, 0,0)$ \ $\big($resp., ${\rm IDD}(K^0_n, 0,0) \big)$ algebras, which are the well-known algebras of restricted polynomials\footnote{Also known as null-filiform associative algebras.} \ $\big($resp., restricted polynomials with unit$\big)$.

\noindent Let $\mathfrak{n} \in \mathbb N \cup \{\infty\}.$ The multiplication tables of algebras of rank $1$ 
are:  
\begin{longtable}{|rrrlrcll|}
\hline
${\rm IDD}(K^0_\mathfrak{n},$ & $1,$ & $0)$ & $:$ & $e_{i}\diamond e_{j}$ & $=$ & $ie_{i+j-1},$ & $1 \le i+j\le \mathfrak{n}+1$ \\
\hline
${\rm IDD}(K^1_\mathfrak{n},$ & $1,$ & $0)$ & $:$ & $e_{i}\diamond e_{j}$ & $=$ & $ie_{i+j-1},$ & $2 \le i+j\le \mathfrak{n}+1$ \\
\hline
${\rm IDD}(K^0_\mathfrak{n},$ & $0,$ & $-1)$ & $:$ & $e_{i}\diamond e_{j}$ & $=$ & $\frac{1}{j+1}e_{i+j+1},$ & $0 \le i+j\le \mathfrak{n}-1$ \\
\hline
${\rm IDD}(K^1_\mathfrak{n},$ & $0,$ & $-1)$ & $:$ & $e_{i}\diamond e_{j}$ & $=$ & $\frac{1}{j+1}e_{i+j+1},$ & $2 \le i+j\le \mathfrak{n}-1$ \\
\hline
\end{longtable}

\noindent The multiplication tables of algebras of rank $2$ are:  
\begin{longtable}{|rrrlrcll|}
\hline
${\rm IDD}(K^0_\mathfrak{n},$ & $2,$ & $0)$ & $:$ & $e_{i}\diamond e_{j}$ & $=$ & $i(i-1)e_{i+j-2},$ & $2 \le i+j\le \mathfrak{n}+2$ \\
\hline
${\rm IDD}(K^1_\mathfrak{n},$ & $2,$ & $0)$ & $:$ & $e_{i}\diamond e_{j}$ & $=$ & $i(i-1)e_{i+j-2},$ & $3 \le i+j\le \mathfrak{n}+2$ \\
\hline
${\rm IDD}(K^0_\mathfrak{n},$ & $1,$ & $1)$ & $:$ & $ e_{i}\diamond e_{j}$ & $=$ & $ij e_{i+j-2},$ & $2 \le i+j\le \mathfrak{n}+2$ \\
\hline
${\rm IDD}(K^1_\mathfrak{n},$ & $1,$ & $1)$ & $:$ & $ e_{i}\diamond e_{j}$ & $=$ & $ij e_{i+j-2},$ & $3 \le i+j\le \mathfrak{n}+2$\\
\hline
${\rm IDD}(K^0_\mathfrak{n},$ & $1,$ & $-1)$ & $:$ & $e_{i}\diamond e_{j}$ & $=$ & $\frac{i}{j+1}e_{i+j},$ & $ 1 \le i+j\le \mathfrak{n}$ \\
\hline
${\rm IDD}(K^1_\mathfrak{n},$ & $1,$ & $-1)$ & $:$ & $e_{i}\diamond e_{j}$ & $=$ & $\frac{i}{j+1}e_{i+j},$ & $ 2 \le i+j\le \mathfrak{n}$ \\
\hline
${\rm IDD}(K^0_\mathfrak{n},$ & $-1,$ & $-1)$ & $:$ & $e_{i}\diamond e_{j}$ & $=$ & $\frac{1}{(i+1)(j+1)}e_{i+j+2},$ & $ 0 \le i+j\le \mathfrak{n}-2$ \\
\hline
${\rm IDD}(K^1_\mathfrak{n},$ & $-1,$ & $-1)$ & $:$ & $e_{i}\diamond e_{j}$ & $=$ & $\frac{1}{(i+1)(j+1)}e_{i+j+2},$ & $ 2 \le i+j\le \mathfrak{n}-2$ \\
\hline
${\rm IDD}(K^0_\mathfrak{n},$ & $0,$ & $-2)$ & $:$ & $e_{i}\diamond e_{j}$ & $=$ & $\frac{1}{(j+1)(j+2)}e_{i+j+2},$ & $ 0 \le i+j\le \mathfrak{n}-2$\\
\hline
${\rm IDD}(K^1_\mathfrak{n},$ & $0,$ & $-2)$ & $:$ & $e_{i}\diamond e_{j}$ & $=$ & $\frac{1}{(j+1)(j+2)}e_{i+j+2},$ & $ 2 \le i+j\le \mathfrak{n}-2$\\
\hline
\end{longtable}

In the general case, we have the following multiplication table. 

\begin{proposition}\label{mtable} Let ${\it l}:=m_1+m_2,$ then the multiplication table of ${\rm IDD}(K^k_\mathfrak{n},m_1, m_2)$ is given below:  
\begin{longtable}{|rclcl|}
\hline
${\rm IDD}(K^k_\mathfrak{n},m_1, m_2)$ & $:$ & $e_{i}\diamond e_{j}$ & $=$ & $\frac{i!}{(i-m_1)!} \frac{j!}{(j-m_2)!}e_{i+j-{\it l}},$ \\
\multicolumn{5}{|r|}{\mbox{\it where} ${\rm max}\big\{k,m_1 \big\}\le i \leq \mathfrak{n},$} \\
\multicolumn{5}{|r|}{${\rm max}\big\{k,m_2\big\} \le j\leq \mathfrak{n},$} \\
\multicolumn{5}{|r|}{\mbox{\it and}\ ${\rm max}\big\{k,k+{\it l}\big\} \le i+j\le \mathfrak{n}+{\it l}.$} \\ 
\hline 
\end{longtable}

\end{proposition}

\begin{proposition}
If $k\geq m_1\geq m_2,$ then ${\rm IDD}(K^k_\infty, m_1, m_2)$ has a strong multiplicative basis.
\end{proposition}

\begin{proof}
It follows from Definitions \ref{mb} and Proposition \ref{mtable}.
\end{proof}

It is obvious that each algebra ${\rm IDD}(K^k_\infty, m_1, m_2)$ is nontrivial.
The finite-dimensional case is more complicated, and our first aim is to identify all nontrivial algebras ${\rm IDD}(K^k_n, m_1, m_2)$ for $m_1\geq m_2.$

\begin{proposition}\label{iddtrivial}
${\rm IDD}(K^k_n, m_1,m_2)$ is trivial if and only if one of the following conditions is true \begin{enumerate}

\item[$({\rm 1})$]  $ m_1  > n;$  

\item[$({\rm 2})$]  $ n \geq m_1$ and ${\it l} > 2n-k;$  

\item[$({\rm 3})$]  $ n \geq m_1$ and ${\it l} < 2k-n.$  

\end{enumerate}
\end{proposition}

\begin{proof}
Firstly, we consider the case $m_1>n.$ Then $m_1>0$ and for each pair $(i,j),$ such that $k \leq i,j \leq n,$ we have $T^{m_1}(e_i)=0$ and $e_i \diamond e_j = T^{m_1}(e_i)T^{m_2}(e_j)= 0,$ i.e., ${\rm IDD}(K^k_n, m_1,m_2)$ is trivial.

Secondly, if $n \geq m_1,$ then for each pair  $i$ and $j,$ such that  $k \leq i,j \leq n,$ we have \begin{center}
$T^{m_1}(e_i) \in \big\langle e_{i-m_1} \big\rangle$ and $T^{m_2}(e_j) \in \big\langle e_{j-m_2} \big\rangle,$ i.e., $T^{m_1}(e_i) T^{m_2}(e_j) \in \big\langle e_{i+j-{\it l}} \big\rangle.$
\end{center}

\begin{enumerate}

\item[${\rm (A)}$] If ${\it l} > 2n-k,$ then $i+j-{\it l} <i+j-2n+k<k,$ i.e., $e_i \diamond e_j =0$ for $k\leq i,j\leq n$ and ${\rm IDD}(K^k_n, m_1,m_2)$ is trivial.

\item[${\rm (B)}$] If ${\it l} < 2k-n,$ then $i+j-{\it l} >i+j-2k+n>n,$ i.e., $e_i \diamond e_j =0$ for $k\leq i,j\leq n$ and ${\rm IDD}(K^k_n, m_1,m_2)$ is trivial.

\end{enumerate}

\noindent On the other hand, if $2k-n \leq {\it l} \leq 2n-k,$ we have the following two cases. 

\begin{enumerate}

\item[${\rm (A)}$] If ${\it l}=2n-t,$ where $k\leq t \leq n,$ then $0 \neq e_n \diamond e_n \in \big\langle e_t \big\rangle,$ i.e., ${\rm IDD}(K^k_n, m_1,m_2)$ is nontrivial.

\item[${\rm (B)}$] If $2k-n\leq {\it l} < n,$ then taking $i= \big[\frac{{\it l}+n}{2} \big]$ we have that $0 \neq e_i \diamond e_i \in \big\langle e_{n-1}, e_n \big\rangle,$ i.e., ${\rm IDD}(K^k_n, m_1,m_2)$ is nontrivial.
\qedhere
\end{enumerate}
\end{proof}

\noindent From now on, we are only interested in nontrivial ${\rm IDD}$ algebras, i.e., ${\rm IDD}(K^k_\mathfrak{n}, m_1,m_2),$ such that \begin{center}
\begin{longtable}{|c|}
\hline
$\infty > \mathfrak{n} \geq m_1$ and $2k-\mathfrak{n} \leq  m_1+m_2 \leq 2\mathfrak{n}-k,$ or $\mathfrak{n}=\infty.$ \\
\hline
\end{longtable}
\end{center}

\subsection{Simple, perfect and nilpotent \texorpdfstring{${\rm IDD}$}{IDD} algebras}\label{spn}

\begin{proposition}
If $n  <  {\rm lev} \big({\rm IDD}(K^k_{n}, m_1, m_2)\big) \leq 2n-k,$ then ${\rm IDD}(K^k_{n},m_1, m_2)$ is nilpotent.
\end{proposition}

\begin{proof}
Let us denote ${\rm A}:={\rm IDD}(K^k_{n}, m_1, m_2)$ and $l:={\rm lev} \big({\rm IDD}(K^k_{n}, m_1, m_2)\big),$ then \begin{center}
${\rm A} \diamond {\rm A} \subseteq \big\langle e_k, \ldots, e_{n+(n-l)} \big\rangle  \subset {\rm A};$\\ 
${\rm A} \diamond \big({\rm A} \diamond {\rm A}\big) + \big({\rm A}\diamond {\rm A} \big) \diamond {\rm A}  \subseteq \big\langle e_k, \ldots, e_{n+2(n-l)} \big\rangle \subset {\rm A} \diamond {\rm A}.$\\
\end{center} Applying mathematical induction, we find that
\begin{center}
${\rm A}^{\diamond t}: = \sum\limits_{i+j=t} {\rm A}^{\diamond i} \diamond {\rm A}^{\diamond j} \subset \big\langle e_k, \ldots, e_{n+(t-1)(n-l)} \big\rangle.$ 
\end{center} Due to $n<l,$ there exists $n_0,$ such that ${\rm A}^{\diamond n_0}=0,$ i.e., ${\rm A}$ is nilpotent.
\end{proof}

\begin{proposition}
If $ {\rm lev} \big({\rm IDD}(K^k_{n}, m_1, m_2)\big) = n,$ then ${\rm IDD}(K^k_{n},m_1, m_2)$ is non-simple perfect and it has exactly $n-k$ proper ideals.
\end{proposition}

\begin{proof}
Let us denote ${\rm A}:={\rm IDD}(K^k_{n}, m_1, m_2),$ then
\begin{center}
${\rm A} \diamond {\rm A} \supseteq \big\langle e_k, \ldots, e_{n} \big\rangle = {\rm A},$ i.e., ${\rm A}$ is perfect.
\end{center} Let $I$ be an ideal of ${\rm A}$ and $n_0 \in \mathbb N,$ such that there exists $\mathfrak i_0 = \sum\limits_{i=k}^{n_0}\alpha_i e_i \in I,$ where $\alpha_{n_0}\neq 0,$ and there are no elements $\mathfrak i = \sum\limits_{i=k}^{t}\alpha_i e_i \in I,$ where $\alpha_t\neq 0$ and $t>n_0.$ It is easy to see that $I=\big\langle e_k, \ldots, e_{n_0} \big\rangle,$ i.e., there exists a $1$-$1$ correspondence between the set of ideals of ${\rm A}$ and the set of numbers $\{k, \ldots, n\}.$ Considering that the ideal $I=\big\langle e_k, \ldots, e_{n} \big\rangle$ is not proper, we have our statement.
\end{proof}

\begin{proposition}
If $0\leq k < {\rm lev} \big({\rm IDD}(K^k_\mathfrak{n}, m_1, m_2)\big) < \mathfrak{n},$ then ${\rm IDD}(K^k_\mathfrak{n},m_1, m_2)$ is simple.
\end{proposition}

\begin{proof}
Let $I$ be a nonzero ideal of ${\rm A}:={\rm IDD}(K^k_\mathfrak{n}, m_1, m_2)$ and $l:={\rm lev} \big({\rm IDD}(K^k_\mathfrak{n}, m_1, m_2)\big).$ Then there exists an element ${\mathfrak i} \in I,$ such that ${\mathfrak i} \ = \ \sum\limits_{i=k}^{k_0} \alpha_i e_i,$ where $k_0 \geq k$ and $\alpha_{k_0} \neq 0.$ Hence, $0\neq {\mathfrak i}{\rm R}_{e_{l-1}}^{k_0-k} \in \big\langle e_k\big\rangle,$ i.e., $e_k\in I.$ The last gives that ${\rm A} \diamond \big\langle e_k\big\rangle \subseteq I,$ i.e., $\{e_k, \ldots, e_{\mathfrak n-l+k} \} \subseteq I.$ If $\mathfrak n = \infty,$ we have our statement. On the other hand, if $\mathfrak n < \infty,$
we have to mention that $e_{k+1}\in I$ and $0\neq e_{\mathfrak n} \diamond {e_{k+1}} \in \big\langle e_{\mathfrak n-l+k+1} \big\rangle,$ i.e., $e_{k+2}\in I.$ That gives $0\neq e_{\mathfrak n} \diamond {e_{k+2}} \in \big\langle e_{\mathfrak n-l+k+2} \big\rangle,$ i.e., $e_{k+3}\in I$ and so on. At the end, we have that ${\rm A}=I.$ The statement is proved.
\end{proof}

\begin{proposition}
If $ {\rm lev} \big({\rm IDD}(K^k_\mathfrak{n}, m_1, m_2)\big) = k,$
then  ${\rm IDD}(K^k_\mathfrak{n},m_1, m_2)$  is non-simple perfect
and it has exactly $\mathfrak n-k$ proper ideals if $\mathfrak n < \infty$ 
and it has infinitely many proper ideals if $\mathfrak n = \infty$.
\end{proposition}

\begin{proof}
Let us denote ${\rm A}:={\rm IDD}(K^k_\mathfrak{n}, m_1, m_2).$ Firstly, we consider the finite-dimensional case $\mathfrak n = n \in \mathbb N,$ then \begin{center}
${\rm A} \diamond {\rm A} \supseteq \big\langle e_k, \ldots, e_n \big\rangle = {\rm A},$ i.e., ${\rm A}$ is perfect.
\end{center} Let $I$ be an ideal of ${\rm A}$ and $n_0 \in \mathbb N,$ such that there exists $\mathfrak i_0 = \sum\limits_{i=n_0}^{n}\alpha_i e_i \in I,$ where $\alpha_{n_0}\neq 0,$ and there are no elements $\mathfrak i = \sum\limits_{i=t}^{n}\alpha_i e_i \in I,$ where $\alpha_t\neq 0$ and $t<n_0.$ It is easy to see that $I=\big\langle e_{n_0}, \ldots, e_{n} \big\rangle,$ i.e., there exists a $1$-$1$ correspondence between the set of ideals of ${\rm A}$ and the set of numbers $\{k, \ldots, n\}.$ Considering that the ideal $I=\big\langle e_k, \ldots, e_{n} \big\rangle$ is not proper, we have our statement.

Secondly, we consider the infinite-dimensional case $\mathfrak n =\infty,$ then \begin{center}
${\rm A} \diamond {\rm A} \supseteq \big\langle e_k, \ldots, e_n, \ldots \big\rangle = {\rm A},$ i.e., ${\rm A}$ is perfect. \end{center} Each element $e_{n_0}$ for $n_0 \geq k$ generates an ideal $I_{n_0}$ of ${\rm A}.$ Namely, ideals $I_{n_0} =\langle e_{n_0},.., e_n,..\rangle$ are different and proper, hence ${\rm A}$ has infinitely many proper ideals.
\end{proof}

\begin{proposition}
If $2k-n \leq {\rm lev} \big({\rm IDD}(K^k_{n}, m_1, m_2)\big) < k,$ then ${\rm IDD}(K^k_{n},m_1, m_2)$ is nilpotent.
\end{proposition}

\begin{proof}
Let us denote ${\rm A}:={\rm IDD}(K^k_{n}, m_1, m_2)$ and $l:={\rm lev} \big({\rm IDD}(K^k_{n}, m_1, m_2)\big),$
then  \begin{longtable}{rcccl}
${\rm A} \diamond {\rm A} $ & $\subseteq$ & $ \big\langle e_{k+(k-l)}, \ldots, e_{n} \big\rangle $ & $\subset$ & ${\rm A};$\\ 
${\rm A} \diamond \big({\rm A} \diamond {\rm A}\big) + \big({\rm A}\diamond {\rm A} \big) \diamond {\rm A} $ & $ \subseteq $ & $\big\langle e_{k+2(k-l)}, \ldots, e_{n} \big\rangle $ & $ \subset $ & ${\rm A} \diamond {\rm A}.$\\
\end{longtable} \noindent  Applying the mathematical induction, we find that \begin{center}
${\rm A}^{\diamond t}: = \sum\limits_{i+j=t} {\rm A}^{\diamond i} \diamond {\rm A}^{\diamond j} \subset \big\langle e_{k+(t-1)(k-l)}, \ldots, e_{n} \big\rangle.$
\end{center} Due to $l<k,$ there exists $n_0,$ such that ${\rm A}^{\diamond n_0}=0,$ i.e., ${\rm A}$ is nilpotent.
\end{proof}

\begin{definition}[see \cite{B13}]
An infinite-dimensional algebra ${\rm A}$ is called pro-nilpotent if \begin{center}
$\bigcap_{n=1}^\infty {\rm A}^n = \big\{0\big\},$ where ${\rm A}^n:=\sum\limits_{i+j=n} {\rm A}^{i}{\rm A}^{j}.$ \end{center}
\end{definition}

\begin{proposition}
If ${\rm lev} \big({\rm IDD}(K^k_{\infty}, m_1, m_2)\big) < k,$ then ${\rm IDD}(K^k_{\infty},m_1, m_2)$ is pro-nilpotent.
\end{proposition}

\begin{proof}
Let us denote ${\rm A}:={\rm IDD}(K^k_{\infty}, m_1, m_2)$ and $l:={\rm lev} \big({\rm IDD}(K^k_{\infty}, m_1, m_2)\big),$ then  \begin{longtable}{rcccl}
${\rm A} \diamond {\rm A} $ & $=$ & $ \big\langle e_{k+(k-l)}, \ldots, e_n, \ldots \big\rangle $ & $ \subset$ & ${\rm A};$\\ 
${\rm A} \diamond \big({\rm A} \diamond {\rm A}\big) + \big({\rm A}\diamond {\rm A} \big) \diamond {\rm A} $ & $ =$ & $ \big\langle e_{k+2(k-l)}, \ldots, e_{n}, \ldots \big\rangle $ & $ \subset$ & $ {\rm A} \diamond {\rm A}.$\\
\end{longtable}\noindent  Applying mathematical induction, we find that \begin{center}
${\rm A}^{\diamond t}: = \sum\limits_{i+j=t} {\rm A}^{\diamond i} \diamond {\rm A}^{\diamond j} = \big\langle e_{k+(t-1)(k-l)}, \ldots, e_{n}, \ldots \big\rangle.$
\end{center} Due to $l<k,$ we have that ${\rm A}^{\diamond t_1} \subset {\rm A}^{\diamond t_2},$ if $t_1>t_2.$ It is easy to see that $\bigcap_{n=1}^\infty {\rm A}^{\diamond n} = \big\{0\big\},$ i.e., ${\rm A}$ is pro-nilpotent. 
\end{proof}

\subsection{Conservative \texorpdfstring{${\rm IDD}$}{IDD} algebras}\label{cons}

In 1972, Kantor introduced conservative algebras as a generalization of Jordan algebras $\big($also, see surveys about the study of conservative algebras and superalgebras \cite{P20,k24}$\big).$

\begin{definition}[see \cite{K72}]
A vector space ${\rm V}$ with a multiplication $\cdot$ is called a {\it conservative algebra} if there is a new multiplication $*:{\rm V}\times {\rm V}\rightarrow {\rm V}$ such that \begin{center}
$\big[L_b^\cdot, [L_a^\cdot,\cdot]\big]\ = \ -\big[L_{a*b}^\cdot,\cdot\big],$ \ where \ $\big[L_c^\cdot, {\rm F}\big] (x,y)\ = \ c\cdot{\rm F}(x, y)-{\rm F}(c\cdot x,  y)-{\rm F}(x, c\cdot y).$
\end{center}
In other words, the following identity holds for all $a,b,x,y\in {\rm V}$: \begin{multline}\label{tcons} b\big(a(xy)-(ax)y-x(ay)\big) - a\big((bx)y\big)+\big(a(bx)\big)y+(bx)(ay)-\\
-a\big(x(by)\big)+(ax)(by)+x\big(a(by)\big) \ = \ -(a *b)(xy)+\big((a* b)x\big)y+x\big((a *b)y\big).
\end{multline}
\end{definition}

\begin{proposition}\label{leftcom}
${\rm IDD}(K^k_\infty, m, 0),$ where $m\in \mathbb Z,$ is a left commutative algebra.
\end{proposition}

\begin{proof}
It follows from Proposition \ref{mtable}.
\end{proof}

\begin{definition}[see \cite{KLP15}]
An algebra ${\rm A}$ is called a generalized associative algebra if there exists one multiplication $F$ defined on the same vector space, such that $a(bx) \ = \ F(a,b)x.$
\end{definition}

\begin{theorem}
For any $m\in \mathbb Z,$ ${\rm IDD}(K^k_\infty, m, 0)$ is conservative with an additional multiplication given by \begin{center}
$e_i * e_j= \frac{(i+j-2m)!} {(i+j-m)!} \frac{i!} {(i-m)!}\frac{j!} {(j-m)!} e_{i+j-m}.$   
\end{center}
\end{theorem}

\begin{proof}
By a direct computation, we obtain that ${\rm IDD}(K^k_\infty, m, 0)$ is a generalized associative algebra with additional multiplication $e_i * e_j$ given in our statement. By Proposition \ref{leftcom}, ${\rm IDD}(K^k_\infty, m, 0)$ is left commutative. Thanks to \cite{KLP15}, each left-commutative generalized associative algebra is conservative, hence ${\rm IDD}(K^k_\infty, m, 0)$ is conservative.
\end{proof}

\begin{remark}
Cantarini and Kac classified all complex linearly compact commutative and anticommutative conservative simple superalgebras {\rm \cite{CK10}}. Algebras ${\rm IDD}(K^k_\infty, m, 0),$ where $k<m,$ give a wide class of examples of non-(anti)commutative conservative simple algebras. It is known that ${\rm IDD}(K^0_\infty, 1, 0)$ is a Novikov algebra, and it seems that it firstly appeared in  {\rm \cite{F89}}.
\end{remark}

\section{Derivation of \texorpdfstring{${\rm IDD}$}{IDD} algebras of rank 1}\label{deridd1}

\subsection{Derivations of \texorpdfstring{$\mathrm{IDD}(K^{\times}_{\mathfrak{n}},1,0)$}{IDD(Kx\_n,1,0)}}\label{der_1_0}

\begin{theorem}
If $1\le n < \infty,$ then $\mathfrak{Der}\left({\rm IDD}(K^0_n,1,0)\right) = \left\langle \varphi, \phi \right\rangle,$ where \begin{center}
$\varphi(e_i)= (i-1) e_{i}$ and $\phi(e_i)=  i e_{i-1}.$
\end{center}
\end{theorem}

\begin{proof}
Let $D \in \mathfrak{Der}\left({\rm IDD}(K^0_n,1,0)\right),$ then we can suppose that \begin{center}
$D(e_{0})=\sum\limits_{i=0}^{n}\alpha_{i}e_{i}, \ D(e_{1})=\sum\limits_{i=0}^{n}\beta_{i}e_{i},$ and $D(e_{2})=\sum\limits_{i=0}^{n}\gamma_{i}e_{i}.$
\end{center} Firstly, we have 
\begin{longtable}{lcl}
$0$    &$=$&
$D(e_{0}\diamond e_{0}) \ = \  D(e_{0})\diamond e_{0}+e_{0}\diamond D(e_{0})$ \\
 
    &$=$&$ \left(\sum\limits_{i=0}^{n}\alpha_{i}e_{i}\right)\diamond e_{0}+e_{0}\diamond\left(\sum\limits_{i=0}^{n}\alpha_{i}e_{i}\right) 
    \ = \  \sum\limits_{i=1}^{n}i\alpha_{i}e_{i-1};$
\end{longtable}
\noindent that immediately gives $\alpha_{k}=0$ for $1\le k\le n$ and $D(e_{0})=\alpha_{0}e_{0}.$ Secondly, 
\begin{longtable}{lcl}
$D(e_{0})$ 
    &$=$&$ D(e_{1}\diamond e_{0}) \ = \  D(e_{1})\diamond e_{0}+e_{1}\diamond D(e_{0})$ \\
    &$=$&$ 
    \left(\sum\limits_{i=0}^{n}\beta_{i}e_{i}\right)\diamond e_{0}+e_{1}\diamond\left(\alpha_{0}e_{0}\right) \ 
     = \  \sum\limits_{i=0}^{n}i\beta_{i}e_{i-1}+\alpha_{0}e_{0}  $\\ 
     &$=$&$\left(\alpha_{0}+\beta_{1}\right)e_{0}+\sum\limits_{i=2}^{n}i\beta_{i}e_{i-1};$
\end{longtable}\noindent 
that immediately gives $\beta_{k}=0$ for $1\le k\le n,$ i.e., $D(e_{1})=\beta_{0}e_{0}.$ 
Thirdly,
\begin{longtable}{lcl}
$2D(e_{1})$
&$=$&$D(e_{2}\diamond e_{0}) \ =\  D(e_{2})\diamond e_{0}+e_{2}\diamond D(e_{0})$ \\
&$=$ &
$\left(\sum\limits_{i=0}^{n}\gamma_{i}e_{i}\right)\diamond e_{0}+e_{2}\diamond\left(\alpha_{0}e_{0}\right)$ \\
&$=$&$ \sum\limits_{i=0}^{n}i\gamma_{i}e_{i-1}+2\alpha_{0}e_{1} \ =\  \gamma_{1}e_{0}+2(\alpha_{0}+\gamma_{2})e_{1}+\sum\limits_{i=3}^{n}i\gamma_{i}e_{i-1}.$
\end{longtable}

\noindent that immediately gives $\gamma_{1}=2\beta_{0}, \ \gamma_{2}=-\alpha_{0}, \ \gamma_{k}=0,$ for $  3\le k\le n,$ i.e., \begin{center} $D(e_{2})\ =\ \gamma_{0}e_{0}+2\beta_{0}e_{1}-\alpha_{0}e_{2}.$ \end{center} \noindent Next,  \begin{longtable}{lcl}

$D(e_{2}) \ = \ \frac 12 D(e_{2}\diamond e_{1})\ =\ \frac 12 \left( \left(\gamma_{0}e_{0}+2\beta_{0}e_{1}-\alpha_{0}e_{2}\right)\diamond e_{1}+e_{2}\diamond\left(\beta_{0}e_{0}\right) \right)\ = \

2\beta_{0}e_{1}-\alpha_{0}e_{2};$ \end{longtable}  \noindent that immediately gives $\gamma_{0}=0,$ i.e., $D(e_{0})=\alpha_{0}e_{0}, \ D(e_{1})=\beta_{0}e_{0}, \ D(e_{2})=2\beta_{0}e_{1}-\alpha_{0}e_{2}.$

It is easy to see that $\varphi(e_i)= (i-1) e_{i}$ and $\phi(e_i)= i e_{i-1}$ are derivations of ${\rm IDD}(K^0_n,1,0)$:
\begin{longtable}{lcl}
$\varphi(e_i \diamond e_j)$
    &$=$&$ i(i+j-2) e_{i+j-1} \ = \ i(i-1) e_{i+j-1} + i(j-1) e_{i+j-1}$ \\
    &$=$&$(i-1) e_i \diamond e_j + (j-1)  e_i \diamond e_j \ =\  \varphi(e_i) \diamond e_j + e_i \diamond  \varphi( e_j)$
\end{longtable}
and
\begin{longtable}{lcl}
$\phi(e_i \diamond e_j)$
    &$=$&$ i (i+j-1) e_{i+j-2} \
=\  i(i-1) e_{i+j-2} + ij e_{i+j-2}$ \\
    &$=$&$ \big(i e_{i-1} \big) \diamond e_j + e_i \diamond \big( j e_{j-1}\big) \
    \ =\  \varphi(e_i) \diamond e_j + e_i \diamond  \varphi( e_j).$
\end{longtable}
Then replacing $D$ by $D+\alpha_0\varphi-\beta_0\phi$ we can suppose that $D(e_1) = D(e_2) = 0$.

Furthermore, using the induction method, we prove that $D(e_{k})=0$ for $k\ge 2.$ To do this, let us consider \begin{center}
$D(e_{k+1}) \ = \ \frac 12 D(e_{2}\diamond e_{k}) \ = \ \frac 12\big(D(e_{2})\diamond e_{k}+e_ {2}\diamond D(e_{k})\big) \ = \ 0.$ \end{center} It follows that each derivation of ${\rm IDD}(K^0_n,1,0)$ is a linear combination of $ \varphi$ and $\phi$, that completes the proof of the statement.
\end{proof}

\begin{corollary}
$\mathfrak{Der}\left({\rm IDD}(K^0_\infty,1,0)\right) = \left\langle \varphi, \phi \right\rangle,$ where $\varphi(e_i)=  (i-1) e_{i} $ and $\phi(e_i)=  i e_{i-1}.$
\end{corollary}

\begin{theorem}
If $1\le n < \infty,$ then $\mathfrak{Der}\big({\rm IDD}(K^1_n,1,0)\big) = \big\langle \varphi \big\rangle,$ where $\varphi(e_i)= (i-1) e_i.$
\end{theorem}

\begin{proof}
Let $D \in \mathfrak{Der}\big({\rm IDD}(K^1_n,1,0)\big),$ then $D(e_{1})=\sum\limits_{i=1}^{n}\alpha_{i}e_{i}$ and $D(e_{2})=\sum\limits_{i=1}^{n}\beta_{i}e_{i}.$ Firstly, \begin{center}
$D\left(e_{1}\right)\ = \ D\left(e_{1}\diamond e_{1}\right)\ = \ D\left(e_{1}\right)\diamond e_{1}+e_{1}\diamond \left(e_{1}\right) \ = \ \sum\limits_{i=1}^{n}\left(i+1\right)\alpha_{i}e_{i},$
\end{center} that immediately gives $\alpha_{k}=0 $ for $1\le k\le n$ and $D\left(e_{1}\right)=0.$ Secondly, we have
\begin{longtable}{lcl}
    $2\sum\limits_{i=1}^{n}\beta_{i}e_{i}
    $&$=$&$2D(e_2)\ =\ D(e_2\diamond e_1) \ 
     =\  D(e_2)\diamond e_1+e_2\diamond D(e_1)$ \\
    &$=$ &$\left(\sum\limits_{i=1}^{n} \beta_i e_i\right)e_1 \ =\  \sum\limits_{i=1}^{n} i \beta_i e_i,$
\end{longtable}\noindent 
that immediately gives $\beta_1=0,$ $\beta_k=0$ for $3\le k\le n$ and $D(e_2) = \beta_2 e_2.$ 

It is easy to see that $\varphi(e_i)= (i-1) e_{i}$ is a derivation of ${\rm IDD}(K^1_n,1,0):$  \begin{longtable}{lcl}
$\varphi(e_i \diamond e_j) $ & $=$ & $ i(i+j-2) e_{i+j-1} \ = \ i(i-1) e_{i+j-1} + i(j-1) e_{i+j-1} \ = $\\
& $=$ & $ (i-1) e_i \diamond e_j + (j-1)e_i \diamond e_j \ = \ \varphi(e_i) \diamond e_j + e_i \diamond  \varphi( e_j).$
\end{longtable} 

\noindent Then replacing $D$ by $D-\beta_2\varphi$ we can suppose that $D(e_1)\ =\ D(e_2) \ = \ 0.$ Furthermore, using the induction method, we prove that $D(e_{k})=0$ for $k\ge 2.$ To do this, let us consider \begin{center}
$D(e_{k+1}) \ = \ \frac 12 D(e_{2}\diamond e_{k}) \ = \ \frac 12\big(D(e_{2})\diamond e_{k}+e_ {2}\diamond D(e_{k})\big) \ = \ 0.$ \end{center} It follows that each derivation of ${\rm IDD}(K^1_n,1,0)$ lies in the linear span of $\varphi$, which completes the proof.
\end{proof}

\subsection{Derivations of \texorpdfstring{$\mathrm{IDD}(K^\times_\mathfrak{n},0,-1)$}{IDD(K\^x\_n,0,-1)}}\label{der0-1}

\begin{theorem}
If $1\le n < \infty,$ then $\mathfrak{Der}\big({\rm IDD}(K^0_n,0,-1)\big) = \big\langle \varphi_i\big\rangle_{ 0\le i \le n},$ where \begin{center} $\varphi_i(e_k) = \left(i+1+k\right) e_{k+i},$ for $k \leq n-i.$ \end{center}
\end{theorem}

\begin{proof}
Let $D \in  \mathfrak{Der}\big({\rm IDD}(K^0_n,0,-1)\big),$ then we can say that $D(e_{0}) \ = \ \sum\limits_{i=0}^{n}\alpha_{i}e_{i}.$ Firstly,  \begin{longtable}{lcl}
$D(e_{1})$ & $=$ & $D(e_{0}\diamond e_{0}) \ = \ \left(\sum\limits_{i=0}^{n}\alpha_{i}e_{i}\right)\diamond e_{0}+e_{0}\diamond\left(\sum\limits_{i=0}^{n}\alpha_{i}e_{i}\right)\ = $ \\

& $=$ & $\sum\limits_{i=0}^{n-1}\alpha_{i}e_{i+1}+\sum\limits_{i=0}^{n-1}\frac{1}{i+1}\alpha_{i}e_{i+1}\ = \ \sum\limits_{i=0}^{n-1}\frac{i+2}{i+1}\alpha_{i}e_{i+1}.$
\end{longtable}\noindent It is easy to see that $\varphi_i(e_k)=\left(i+k+1\right) e_{i+k}$ is a derivation of ${\rm IDD}(K^0_n,0,-1):$  \begin{longtable}{lcl}
$\varphi_k(e_i \diamond e_j) $ & $ = $ & $ \frac{k+i+j+2}{j+1} e_{i+j+k+1} \ = \ \frac{i+k+1}{j+1} e_{i+j+k+1} +
\frac{j+k+1}{j+k+1} e_{i+j+k+1} \ = $\\

& $=$ & $ \left(i+k+1\right) e_{i+k} \diamond e_j+ \left(j+k+1\right) e_i \diamond e_{j+k} \ = \ \varphi_k(e_i) \diamond e_j + e_i \diamond \varphi_k(e_j).$
\end{longtable} \noindent  Hence, we can replace $D$ by $D-\sum\limits_{i=0}^n\alpha_i \varphi_i$ and suppose that $D(e_0) \ = \ D(e_1) \ = \ 0.$

Furthermore, using the induction method, we prove that $ D(e_{k})=0$ for $k\ge 2.$ To do this, let us consider
\begin{center} $D(e_{k+1}) \ = \ kD(e_{1}\diamond e_{k-1})\ = \ k \big(D(e_{1})\diamond e_{k-1}+e_{1}\diamond D(e_{k-1}) \big)\ = \ 0.$
\end{center} It follows that each derivation of ${\rm IDD}(K^0_n,0,-1)$ is a linear combination of $\varphi_i$, that completes the proof of the statement.
\end{proof}

\begin{corollary}
$\mathfrak{Der}\big({\rm IDD}(K^0_\infty,0,-1)\big) = \big\langle \varphi_i\big\rangle_{ i \ge 0},$ where $\varphi_i(e_k)= \left(i+1+k\right) e_{k+i}.$
\end{corollary}

\begin{theorem}
If $3\le n < \infty,$ then $\mathfrak{Der}\big({\rm IDD}(K^1_n,0,-1)\big) = \big\langle \varphi_i, \phi_1, \phi_2 \big\rangle_{ 1\le i \le n},$ where  \begin{longtable}{|l|lcll|}
\hline
$\varphi_i$ & $\varphi_i(e_k)$ & $=$ & $\left(i+k\right) e_{k+i-1},$ & $k \leq 1+n-i$ \\
\hline
$\phi_1$ & $\phi_1(e_2)$ & $=$ & $e_{n-1}$ & \\
\hline
$\phi_2$ & $\phi_2(e_2)$ & $=$ & $e_{n}$ & \\
\hline 
\end{longtable}
 
\end{theorem}

\begin{proof}
Let $D \in  \mathfrak{Der}\big({\rm IDD}(K^1_n,0,-1)\big),$ then we can say that \begin{center}
$D(e_{1})=\sum\limits_{i=1}^{n}\alpha_{i}e_{i}$ and $D(e_{2})=\sum\limits_{i=1}^{n}\beta_{i}e_{i}.$
\end{center} Firstly,  \begin{longtable}{lcl}
$D(e_{3})$ & $=$ & $2D(e_{1}\diamond e_{1}) \ = \ 2 \left(\left(\sum\limits_{i=1}^{n}\alpha_{i}e_{i}\right)\diamond e_{1}+e_{1}\diamond\left(\sum\limits_{i=1}^{n}\alpha_{i}e_{i}\right) \right)\ = $ \\

& $=$ & $\sum\limits_{i=1}^{n-2}\alpha_{i}e_{i+2}+\sum\limits_{i=1}^{n-2}\frac{2}{i+1}\alpha_{i}e_{i+2}\ = \ \sum\limits_{i=1}^{n-2}\frac{i+3}{i+1}\alpha_{i}e_{i+2}.$

\end{longtable}  \noindent Secondly,   \begin{longtable}{lcl}
$D(e_{4})$ & $=$ & $3 D(e_{1} \diamond e_{2})\ = \ 3\left(\left(\sum\limits_{i=1}^{n}\alpha_{i}e_{i}\right)\diamond e_{2}+e_{1}\diamond\left(\sum\limits_{i=1}^{n}\beta_{i}e_{i}\right) \right) \ = $\\

& $=$ & $\sum\limits_{i=1}^{n-3}\alpha_{i}e_{i+3}+\sum\limits_{i=1}^{n-2}\frac{3}{i+1}\beta_{i}e_{i+2} \ = \ \frac{3}{2}\beta_{1}e_{3}+\sum\limits_{i=1}^{n-3}\left( \alpha_{i}+\frac{3}{i+2}\beta_{i+1}\right)e_{i+3}.$
\end{longtable}
 \noindent Thirdly,   \begin{longtable}{lcl}
$D(e_{4})$ & $=$ & $2 D(e_{2} \diamond e_{1}) \ = \ 2\left(\left(\sum\limits_{i=1}^{n}\beta_{i}e_{i}\right)\diamond e_{1}+e_{2}\diamond\left(\sum\limits_{i=1}^{n}\alpha_{i}e_{i}\right) \right) \ = $ \\

& $=$ & $\sum\limits_{i=1}^{n-2} \beta_{i}e_{i+2}+\sum\limits_{i=1}^{n-3}\frac{2}{i+1}\alpha_{i}e_{i+3} \ = \ \beta_{1}e_{3}+\sum\limits_{i=1}^{n-3}\left(\beta_{i+1}+\frac{2}{i+1}\alpha_{i}\right)e_{i+3}.$
\end{longtable} \noindent Hence, comparing these two expressions for $D(e_4),$ we have that \begin{center}
$\beta_{1}=0$ and $\beta_{k}=\frac{k+1}{k}\alpha_{k-1}, \ \ 3\le k\le n-2,$ i.e.,  \\
$D(e_{2})\ =\ \beta_{2}e_{2}+\sum\limits_{i=3}^{n-2}\frac{i+1}{i}\alpha_{i-1}e_{i}+\beta_{n-1}e_{n-1}+\beta_{n}e_{n}$ and \\ $D(e_{4})\ = \ \left(\alpha_{1}+\beta_{2}\right)e_{4}+\sum\limits_{i=2}^{n-3}\frac{i+4}{i+1}\alpha_{i}e_{i+3}.$
\end{center} Now, we will consider two different expressions of $D(e_5):$   \begin{longtable}{lcl}
$D(e_{5})$ & $=$ & $4D(e_{1} \diamond e_{3})\ = \ 4\left(\left(\sum\limits_{i=1}^{n}\alpha_{i}e_{i}\right)\diamond e_{3}+e_{1}\diamond\left(\sum\limits_{i=1}^{n-2}\frac{i+3}{i+1}\alpha_{i}e_{i+2}\right) \right) \ =$ \\

& $=$ & $\sum\limits_{i=1}^{n-4} \alpha_{i}e_{i+4}+\sum\limits_{i=1}^{n-4}\frac{4}{i+1}\alpha_{i}e_{i+4}\ = \ 
3 \alpha_1 e_5+\sum\limits_{i=2}^{n-4}\frac{i+5}{i+1}\alpha_{i}e_{i+4};$\\
$D(e_{5})$&$=$&$ 3D(e_{2}\diamond e_{2}) $\\
&$=$&$ 3\Bigg(\left(\beta_{2}e_{2}+\sum\limits_{i=3}^{n-2}\frac{i+1}{i}\alpha_{i-1}e_{i}
+\beta_{n-1}e_{n-1}+\beta_{n}e_{n}\right)\diamond e_{2}$ \\
& &$\qquad\quad +\, e_{2}\diamond\left(\beta_{2}e_{2}
+\sum\limits_{i=3}^{n-2}\frac{i+1}{i}\alpha_{i-1}e_{i}
+\beta_{n-1}e_{n-1}+\beta_{n}e_{n}\right)\Bigg)$ \\
&$=$&$ \beta_{2}e_{5}+\sum\limits_{i=3}^{n-3}\frac{i+1}{i}\alpha_{i-1}e_{i+3}
+\beta_{2}e_{5}+\sum\limits_{i=3}^{n-3}\frac{3}{i}\alpha_{i-1}e_{i+3} $\\
&$=$&$ 2\beta_{2}e_{5}+\sum\limits_{i=2}^{n-4}\frac{i+5}{i+1}\alpha_{i}e_{i+4}.$
\end{longtable}

\noindent The last gives $\beta_{2}=\frac{3}{2}\alpha_{1},$ i.e. \begin{center} $D(e_{2})\ =\ \sum\limits_{i=1}^{n-3}\frac{i+2}{i+1}\alpha_{i}e_{i+1}+\beta_{n-1}e_{n-1}+\beta_{n}e_{n}$ and
$D(e_{4})\ =\ \sum\limits_{i=1}^{n-3}\frac{i+4}{i+1}\alpha_{i}e_{i+3}.$ \end{center}

Let us consider linear mappings $\phi_1(e_2) = e_{n-1}$ \ and \ $\phi_2(e_2) = e_{n}.$ Obviously, \begin{center}
$\phi_1\big({\rm IDD}(K^1_n,0,-1)\big) + \phi_2\big({\rm IDD}(K^1_n,0,-1)\big) \subseteq {\rm Ann}\big( {\rm IDD}(K^1_n,0,-1) \big),$ \end{center} and $e_2 \notin \big({\rm IDD}(K^1_n,0,-1) \big)^2,$ hence $\phi_1$ \ and \ $\phi_2$ are derivations. It is easy to see that linear mappings $\varphi_i$ defined by \begin{center} $\varphi_i(e_k)= \left(i+k\right) e_{k+i-1},$  where $k \leq 1+n-i$ \\ \end{center} are also derivations of ${\rm IDD}(K^1_n,0,-1):$  \begin{longtable}{lcl}
$\varphi_k (e_i \diamond e_j)$ & $=$ & $\frac{k+i+j+1}{k+1} e_{i+k+j} \ = \ \frac{k+i}{j+1} e_{k+i+j} + \frac{k+j}{k+j} e_{k+i+j} \ =$ \\

& $=$ & $\left(k+i\right) e_{k+i-1} \diamond e_j + \left(k+j\right)e_i \diamond e_{k+j-1} \ = \ \varphi_k (e_i) \diamond e_j + e_i \diamond \varphi_k (e_j).$
\end{longtable}   \noindent Hence, we can replace $D$ by $D-\sum\limits_{i=1}^n\alpha_i \varphi_i - \left(\beta_{n-1}-n \alpha_{n-2}\right)\phi_1 - \left(\beta_{n}-\left(n-1\right) \alpha_{n-1}\right)\phi_2,$ and suppose that $D(e_1) \ = \ D(e_2) \ = \ D(e_3) \ = \ D(e_4) \ = \ 0.$

Furthermore, using the induction method, we prove that $ D(e_{k})=0$ for $k\ge 3.$ To do this, let us consider \begin{center} $D(e_{k+1}) \  = \ kD(e_{1}\diamond e_{k-1})\ = \ k \big(D(e_{1})\diamond e_{k-1}+e_{1}\diamond D(e_{k-1}) \big)\ =\ 0.$ \end{center} It follows that each derivation of ${\rm IDD}(K^1_n,0,-1)$ is a linear combination of $\varphi_i,$ $\phi_1$ and $\phi_2$, that completes the proof of the statement.
\end{proof}

\begin{corollary}
$\mathfrak{Der}\big({\rm IDD}(K^1_\infty,0,-1)\big) = \big\langle \varphi_i \big\rangle_{ i\ge 1},$ where $\varphi_i(e_k)= \left(i+k\right) e_{k+i-1}.$
\end{corollary}

\section{Derivation of \texorpdfstring{$\mathrm{IDD}$}{IDD} algebras of rank 2}\label{deridd2}

\subsection{Derivations of \texorpdfstring{$\mathrm{IDD}(K^\times_\mathfrak{n},2,0)$}{IDD(K\^x\_n,2,0)}}\label{der1-10}

\begin{proposition} $\mathfrak{Der}\big({\rm IDD}(K^0_2,2,0)\big) = \big\langle \varphi_i, \phi_i, \psi_i  \big\rangle_{0\le i\le 1}$ where   \begin{longtable}{|c|c|l|}
\hline
$\varphi_{i}$ & $\varphi_{i}(e_0)=e_{i}$ & $0\le i\le 1$ \\
\hline
$\phi_{i}$ & $\phi_{i}(e_1)=e_{i}$ & $0\le i\le 1$ \\
\hline
$\psi_{i}$ & $\psi_{i}(e_2)=e_{i}$ & $0\le i\le 1$ \\
\hline
\end{longtable}
 
\end{proposition}

\begin{theorem}
If $3\le n < \infty,$ then $\mathfrak{Der}\big({\rm IDD}(K^0_n,2,0)\big) = \big\langle \varphi, \phi \big\rangle,$ where \begin{center}
$\varphi(e_i)= (i-2) e_i$ and $\phi(e_i)= i e_{i-1}.$  
\end{center}
\end{theorem}

\begin{proof}
Let $D \in \mathfrak{Der}\big({\rm IDD}(K^0_n,2,0)\big),$ then we can suppose that \begin{center} $D(e_{0})=\sum\limits_{i=0}^{n}\alpha_{i}e_{i}, \ D(e_{1})=\sum\limits_{i=0}^{n}\beta_{i}e_{i}, \ D(e_{2})=\sum\limits_{i=0}^{n}\gamma_{i}e_{i}, \ D(e_{3})=\sum\limits_{i=0}^{n}\lambda_{i}e_{i}.$
\end{center} Firstly, \begin{center}
$0\ =\ D(e_{0}\diamond e_{0})\ = \ \left(\sum\limits_{i=0}^{n}\alpha_{i}e_{i}\right)\diamond e_{0}+e_{0}\diamond\left(\sum\limits_{i=0}^{n}\alpha_{i}e_{i}\right)\ = \ \sum\limits_{i=0}^{n}i(i-1)\alpha_{i}e_{i-2},$
\end{center} that immediately gives $\alpha_{k}=0,$ for $2\le k\le n,$ i.e., $D(e_{0})=\alpha_{0}e_{0}+\alpha_{1}e_{1}.$ Secondly, \begin{center}
$0 \ = \ D(e_{1}\diamond e_{0}) \ = \ \left(\sum\limits_{i=0}^{n}\beta_{i}e_{i}\right)\diamond e_{0}+e_{1}\diamond\left(\sum\limits_{i=0}^{n}\alpha_{i}e_{i}\right)\ = \ \sum\limits_{i=0}^{n}i(i-1)\beta_{i}e_{i-2},$
\end{center} that immediately gives $\beta_{k}=0,$ for $2\le k\le n,$ i.e., $D(e_{1})=\beta_{0}e_{0}+\beta_{1}e_{1}.$ Thirdly,  \begin{longtable}{lcl}
$D(e_{0})$ & $ = $ & $ \frac{1}{2}D(e_{2}\diamond e_{0})\ = \ \frac{1}{2}\left(\left(\sum\limits_{i=0}^{n}\gamma_{i}e_{i}\right)\diamond e_{0}+e_{2}\diamond\left(\alpha_{0}e_{0}+\alpha_{1}e_{1}\right) \right) \ = $ \\

& $=$ & $ \frac{1}{2}\sum\limits_{i=0}^{n}i(i-1)\gamma_{i}e_{i-2}+\alpha_{0}e_{0}+\alpha_{1}e_{1},$
\end{longtable}\noindent that immediately gives $\gamma_{k}=0,$ for $2\le k\le n,$ i.e., $D(e_{2})=\gamma_{0}e_{0}+\gamma_{1}e_{1}.$ Next,  \begin{longtable}{lcl}
$D(e_{1})$ & $=$ & $ \frac{1}{6} D(e_{3}\diamond e_{0})\ =\ \frac{1}{6} \left(\left(\sum\limits_{i=0}^{n}\lambda_{i}e_{i}\right)\diamond e_{0}+e_{3}\diamond\left(\alpha_{0}e_{0}+\alpha_{1}e_{1}\right) \right) \ = $ \\

& $=$ & $\frac{1}{3} \lambda_{2}e_{0}+ \left(\alpha_{0}+\lambda_{3}\right)e_{1}+ \left(\alpha_{1}+2\lambda_{4}\right)e_{2}+ \frac{1}{6} \sum\limits_{i=5}^{n}i(i-1)\lambda_{i}e_{i-2},$
\end{longtable} \noindent that immediately gives $\lambda_{2}=3\beta_{0},$ $\lambda_{3}=-\alpha_{0}+\beta_{1},$
$\lambda_{4}=-\frac{1}{2}\alpha_{1},$ and $\lambda_{k}=0,$ for $5\le k\le n$, i.e., \begin{center} $D(e_{3})=\lambda_{0}e_{0}+\lambda_{1}e_{1}+3\beta_{0}e_{2}+\left(-\alpha_{0}+\beta_{1}\right)e_{3}-\frac{1}{2}\alpha_{1}e_{4}.$
\end{center} Finally,  \begin{longtable}{lcl}
$D(e_{2})$ & $=$ & $\frac{1}{6} D(e_{3}\diamond e_{1})\ = \ \frac{1}{6} \big( \left(\lambda_{0}e_{0}+\lambda_{1}e_{1}+3\beta_{0}e_{2}+\left(-\alpha_{0}+\beta_{1}\right)e_{3}-\frac{1}{2}\alpha_{1}e_{4}\right)\diamond e_{1}+$  \\
& & $+e_{3}\diamond\left(\beta_{0}e_{0}+\beta_{1}e_{1}\right) \big) \ = \ \beta_{0}e_{1}+\left(-\alpha_{0}+\beta_{1}\right)e_{2}-\alpha_{1}e_{3}+\beta_{0}e_{1}+\beta_{1}e_{2} \ = $ \\

& $=$ & $2\beta_{0}e_{1}+\left(-\alpha_{0}+\beta_{1}\right)e_{2}-\alpha_{1}e_{3},$
\end{longtable} \noindent  that immediately gives $\gamma_{0}=0,$ \ $\gamma_{1}=2\beta_{0},$ \ $\beta_{1}=\frac{1}{2}\alpha_{0},$ and $\alpha_{1}=0.$  \begin{longtable}{lcl}
$D(e_{3})$ & $=$ & $\frac{1}{6} D(e_{3}\diamond e_{2})\ = \ 
\frac{1}{6} \big(\left(\lambda_{0}e_{0}+\lambda_{1}e_{1}+3\beta_{0}e_{2}-\frac{1}{2}\alpha_{0}e_{3}\right)\diamond e_{2}+e_{3}\diamond\left(2\beta_{0}e_{1}\right) \big) \ = $ \\
& $=$ & $3 \beta_{0}e_{2}-\frac12 \alpha_{0}e_{3},$ \end{longtable} \noindent  that immediately gives $\lambda_{0}=0$ and $\lambda_{1}=0,$ i.e., \begin{center}
$D(e_{0})=\alpha_{0}e_{0},$ \ $D(e_{1})=\beta_{0}e_{0}+\frac{1}{2}\alpha_{0}e_{1},$ \ $D(e_{2})=2\beta_{0}e_{1}$ \ and \ $D(e_{3})=3\beta_{0}e_{2}-\frac{1}{2}\alpha_{0}e_{3}.$ \\
\end{center}

It is easy to see that $\varphi(e_i)=(i-2) e_{i}$ and $\phi(e_i)= ie_{i-1}$ are derivations of ${\rm IDD}(K^0_n,2,0):$   \begin{longtable}{lcl}
$\varphi(e_i \diamond e_j) $ & $=$ & $ i(i-1)(i+j-4) e_{i+j-2} \ = \ i(i-1)\left((i-2) e_{i+j-2} + (j-2) e_{i+j-2}\right) \ = \ $ \\
& $=$ & $\big((i-2) e_i \big) \diamond e_j + e_i \diamond \big((j-2)e_j\big) \ = \ \varphi(e_i) \diamond e_j + e_i \diamond  \varphi( e_j);$ \\
$\phi(e_i \diamond e_j) $ & $=$ & $ i(i-1)(i+j) e_{i+j-3} \ = \ i(i-1)i e_{i+j-3} + i(i-1)j e_{i+j-3} \ = \ $ \\
& $=$ & $\big(i e_{i-1} \big) \diamond e_j + e_i \diamond \big( j e_{j-1}\big) \ = \ \varphi(e_i) \diamond e_j + e_i \diamond \varphi( e_j).$ \end{longtable} \noindent  Then replacing $D$ by $D+\alpha_0\varphi-\beta_0\phi$ we can suppose that $D(e_1)\ = \ D(e_2) \ = \ D(e_3) \ = \ 0.$

Furthermore, using the induction method, we prove that $ D(e_{k})=0$ for $k\ge 2.$ To do this, let us consider \begin{center} $D(e_{k+1}) \ = \ \frac{1}{6} D(e_{3}\diamond e_{k})\ = \ \frac{1}{6}\big(D(e_{3})\diamond e_{k}+e_{3}\diamond D(e_{k})\big) \ = \ 0.$ \end{center} It follows that each derivation of ${\rm IDD}(K^0_n,2,0)$ is a linear combination of $\varphi$ and $\phi$, that completes the proof of the statement.
\end{proof}

\begin{corollary}
$\mathfrak{Der}\big({\rm IDD}(K^0_\infty,2,0)\big) = \big\langle \varphi, \phi \big\rangle,$ where $\varphi(e_i)= (i-2) e_i$ and $\phi(e_i)= i e_{i-1}.$
\end{corollary}

\begin{proposition}
$\mathfrak{Der}\big({\rm IDD}(K^1_2,2,0)\big) = \big\langle \varphi, \phi \big\rangle$ where $\varphi(e_{1})=e_{1}$ and $\phi(e_{2})=e_{1}.$
\end{proposition}

\begin{theorem}
If $3\le n < \infty,$ then $\mathfrak{Der}\big({\rm IDD}(K^1_n,2,0)\big) = \big\langle \varphi \big\rangle,$ where $\varphi(e_i)= (i-2) e_i.$  
\end{theorem}

\begin{proof}
Let $D \in \mathfrak{Der}\big({\rm IDD}(K^1_n,2,0)\big),$ then we can suppose that \begin{center}
$D(e_{1})=\sum\limits_{i=1}^{n}\alpha_{i}e_{i},$ $D(e_{2})=\sum\limits_{i=1}^{n}\beta_{i}e_{i},$ and $D(e_{3})=\sum\limits_{i=1}^{n}\gamma_{i}e_{i}.$
\end{center} Firstly, \begin{center} $0\ = \ D(e_{1}\diamond e_{1}) \ = \ \left(\sum\limits_{i=1}^{n}\alpha_{i}e_{i}\right)\diamond e_{1}+e_{1}\diamond\left(\sum\limits_{i=1}^{n}\alpha_{i}e_{i}\right)=\sum\limits_{i=1}^{n}i(i-1)\alpha_{i}e_{i-1};$
\end{center} that immediately  gives $\alpha_{k}=0$ for $2\le k\le n,$ i.e.,  $D(e_{1})=\alpha_{1}e_{1}.$ Secondly, \begin{center} $D(e_{1})\ =\ \frac{1}{2} D(e_{2}\diamond e_{1}) \ = \ \frac{1}{2} \left(\left(\sum\limits_{i=1}^{n}\beta_{i}e_{i}\right)\diamond e_{1}+e_{2}\diamond\left(\alpha_{1}e_{1}\right) \right) \ = \ \frac{1}{2} \sum\limits_{i=1}^{n}i(i-1)\beta_{i}e_{i-1}+\alpha_{1}e_{1};$
\end{center} that immediately gives $\beta_{k}=0$ for $2\le k\le n,$ i.e.,  \begin{longtable}{lcl}
$D(e_{2})$ & $=$ & $\frac{1}{6} D(e_{3}\diamond e_{1}) \ = \ \frac{1}{6} \left(\left(\sum\limits_{i=1}^{n}\gamma_{i}e_{i}\right)\diamond e_{1}+e_{3}\diamond\left(\alpha_{1}e_{1}\right)\right)$\\
& $=$ & $ \frac 16\sum\limits_{i=1}^{n}i(i-1)\gamma_{i}e_{i-1}+\alpha_{1}e_{2} \ = \ \frac{1}{3} \gamma_{2}e_{1}+\left(\gamma_{3}+\alpha_{1}\right)e_{2}+\frac{1}{6}\sum\limits_{i=4}^{n}i(i-1)\gamma_{i}e_{i-1};$
\end{longtable} \noindent  that immediately gives $\gamma_{2}=3\beta_{1}, \gamma_{3}=-\alpha_{1},$ and $\gamma_{k}=0$ for $4\le k\le n,$ i.e., \begin{center}
$D(e_{3})=\gamma_{1}e_{1}+3\beta_{1}e_{2}-\alpha_{1}e_{3}.$
\end{center} Next,   \begin{longtable}{lcl}
$D(e_{3})$ & $=$ & $\frac{1}{6} D(e_{3}\diamond e_{2}) = \frac{1}{6}\big(\left(\gamma_{1}e_{1}+3\beta_{1}e_{2}-\alpha_{1}e_{3}\right)\diamond e_{2}+e_{3}\diamond\left(\beta_{1}e_{1}\right) \big) = 2\beta_{1}e_{2}-\alpha_{1}e_{3};$
\end{longtable}  \noindent   that immediately gives $\beta_{1}=0$ and $\gamma_{1}=0,$ i.e., $D(e_{1})=\alpha_{1}e_{1}, \ \ D(e_{2})=0, \ \ D(e_{3})=-\alpha_{1}e_{3}.$

It is easy to see that $\varphi(e_i)=(i-2)e_{i}$ is a derivation of ${\rm IDD}(K^1_n,2,0):$  \begin{longtable}{lcl}
$\varphi(e_i \diamond e_j) $ & $=$ & $ i(i-1)(i+j-4) e_{i+j-2} \ = \ i(i-1)\left((i-2) e_{i+j-2} + (j-2) e_{i+j-2}\right) \ = $ \\

& $=$ & $\big((i-2) e_i \big) \diamond e_j + e_i \diamond \big((j-2)e_j\big) \ = \ \varphi(e_i) \diamond e_j + e_i \diamond \varphi(e_j).$
\end{longtable} \noindent   Then replacing $D$ by $D+\alpha_1\varphi$ we can suppose that $D(e_1)\ = \ D(e_2) = \ D(e_3) \ = \ 0.$

Furthermore, using the induction method, we  prove that $ D(e_{k})=0$  for $k\ge 2.$  To do this, let us consider \begin{center}
$D(e_{k+1})\ = \ \frac{1}{6} D(e_{3}\diamond e_{k}) \ = \ \frac{1}{6}\big(D(e_{3})\diamond e_{k}+e_ {3}\diamond D(e_{k})\big)\ =\ 0.$ \end{center} It follows that each derivation of ${\rm IDD}(K^1_n,2,0)$ is a linear combination of $\varphi,$ that completes the proof of the statement.
\end{proof}

\begin{corollary}
$\mathfrak{Der}\big({\rm IDD}(K^1_\infty,2,0)\big) = \big\langle \varphi, \phi \big\rangle,$ where $\varphi(e_i)= (i-2) e_i.$
\end{corollary}

\subsection{Derivations of \texorpdfstring{$\mathrm{IDD}(K^\times_\mathfrak{n},1,1)$}{IDD(K\^x\_n,1,1)}}\label{der10}

\begin{theorem}
If $1\le n < \infty,$ then $\mathfrak{Der}\big({\rm IDD}(K^0_n,1,1)\big) = \big\langle \varphi, \phi \big\rangle,$ where \begin{center}
$\varphi(e_i)= (i-2) e_i$ and $\phi(e_i) = i e_{i-1}.$  
\end{center}
\end{theorem}

\begin{proof}
Let $D \in \mathfrak{Der}\big({\rm IDD}(K^0_n,1,1)\big),$ then we can say that \begin{center} $D(e_{0})=\sum\limits_{i=0}^{n}\alpha_{i}e_{i},$ $D(e_{1})=\sum\limits_{i=0}^{n}\beta_{i}e_{i},$ 
$D(e_{2})=\sum\limits_{i=0}^{n}\gamma_{i}e_{i},$ and $D(e_{3})=\sum\limits_{i=0}^{n}\lambda_{i}e_{i}.$ \end{center} Firstly, \begin{center}
$0 \ = \ D(e_{0}\diamond e_{1}) \ = \ \left(\sum\limits_{i=0}^{n}\alpha_{i}e_{i}\right)\diamond e_{1}+e_{0}\diamond\left(\sum\limits_{i=0}^{n}\beta_{i}e_{i}\right) \ = \ \sum\limits_{i=0}^{n}i\alpha_{i}e_{i-1},$
\end{center} that immediately gives $\alpha_{k}=0,$ for $1\le k\le n,$ i.e., $D(e_{0})=\alpha_{0}e_{0}.$ Secondly, \begin{center}
$D(e_{0}) \ = \ D(e_{1}\diamond e_{1})\ =\ \left(\sum\limits_{i=0}^{n}\beta_{i}e_{i}\right)\diamond e_{1}+e_{1}\diamond\left(\sum\limits_{i=0}^{n}\beta_{i}e_{i}\right) \ = \ \sum\limits_{i=0}^{n}2i\beta_{i}e_{i-1},$
\end{center} that immediately gives $\beta_{1}=\frac{1}{2}\alpha_{0}$ and $\beta_{k}=0,$ for $2\le k\le n,$ i.e., $D(e_{1})=\beta_{0}e_{0}+\frac{1}{2}\alpha_{0}e_{1}.$ Thirdly,  \begin{longtable}{lcl}
$D(e_{1})$ & $=$ & $ \frac{1}{2} D(e_{2}\diamond e_{1}) \ = \ \frac 12 \left(\left(\sum\limits_{i=0}^{n}\gamma_{i}e_{i}\right)\diamond e_{1}+e_{2}\diamond\left(\beta_{0}e_{0}+\frac{1}{2}\alpha_{0}e_{1}\right) \right)$\\   
& $=$ & $\frac 12 \sum\limits_{i=0}^{n}i\gamma_{i}e_{i-1}+\alpha_{0}e_{1} \ = \ \frac{1}{2}\gamma_{1}e_{0}+\left(\frac{1}{2}\alpha_{0}+\gamma_{2}\right)e_{1}+\sum\limits_{i=3}^{n}\frac{i}{2}\gamma_{i}e_{i-1};$ 
\end{longtable} \noindent  that immediately gives $\gamma_{1}=2\beta_{0}$ and $\gamma_{k}=0,$ for $2\le k\le n,$ i.e., $D(e_{2})=\gamma_{0}e_{0}+2\beta_{0}e_{1}.$ Next,  \begin{longtable}{lcl}
$D(e_{2})$ & $ =$ & $ \frac 14 D(e_{2}\diamond e_{2}) \ = \ \frac 14 \left(\left(\gamma_{0}e_{0}+2\beta_{0}e_{1}\right)\diamond e_{2}+e_{2}\diamond\left(\gamma_{0}e_{0}+2\beta_{0}e_{1}\right)\right) \ = \ 2\beta_{0}e_{1};$\\

$D(e_{2})$ & $=$ & $ \frac 13D(e_{3}\diamond e_{1})\ = \ \frac 13 \left( \left(\sum\limits_{i=0}^{n}\lambda_{i}e_{i}\right)\diamond e_{1}+e_{3}\diamond\left(\beta_{0}e_{0}+\frac{1}{2}\alpha_{0}e_{1}\right) \right) \ = $\\

& $=$ & $\frac 13 \left(\sum\limits_{i=0}^{n}i\lambda_{i}e_{i-1}+\frac{3}{2}\alpha_{0}e_{2} \right) \ = \ \frac{\lambda_{1}}{3}e_{0}+\frac 23\lambda_{2}e_{1}+ \left(\frac{1}{2}\alpha_{0}+\lambda_{3}\right)e_{2}+\sum\limits_{i=4}^{n}\frac{i}{3}\lambda_{i}e_{i-1};$
\end{longtable}  \noindent that immediately gives $\gamma_{0}=0,$ \  $\lambda_{1}=0,$ \ $\lambda_{2}=3\beta_{0},$ \ $\lambda_{3}=-\frac{1}{2}\alpha_{0},$ and $\lambda_{k}=0,$ for $4\le k\le n.$
Finally,  \begin{longtable}{lcl}
$D(e_{3})$ & $=$ & $\frac 16 D(e_{3}\diamond e_{2}) \ = \ \frac 16 \left(\left(\lambda_{0}e_{0}+3\beta_{0}e_{2}-\frac{1}{2}\alpha_{0}e_{3}\right)\diamond e_{2}+e_{3}\diamond\left(2\beta_{0}e_{1}\right)\right) \ = $\\
& $=$ & $3\beta_{0}e_{2}-\frac 12 \alpha_{0}e_{3},$ 
\end{longtable}\noindent  i.e., $\lambda_{0}=0.$ Then we obtain \begin{center} $D(e_{0})=\alpha_{0}e_{0},$ \ $D(e_{1})=\beta_{0}e_{0}+\frac{1}{2}\alpha_{0}e_{1},$ \
$D(e_{2})=2\beta_{0}e_{1}$ \ and \ $D(e_{3})=3\beta_{0}e_{2}-\frac{1}{2}\alpha_{0}e_{3}.$
\end{center}
It is easy to see that $\varphi(e_i)=(i-2)e_{i}$ and $\phi(e_i)=ie_{i-1}$ are derivations of ${\rm IDD}(K^0_n,1,1):$  \begin{longtable}{lcl}
$\varphi(e_i \diamond e_j)$ & $=$ & $ij(i+j-4) e_{i+j-2} \ = \ ij(i-2) e_{i+j-2} + ij(j-2) e_{i+j-2} \ = $\\

& $=$ & $\big((i-2) e_i \big) \diamond e_j + e_i \diamond \big((j-2)e_j\big) \ = \
\varphi(e_i) \diamond e_j + e_i \diamond \varphi( e_j).$\\

$\phi(e_i \diamond e_j)$ & $=$ & $ij(i+j) e_{i+j-3} \ = \ iji e_{i+j-3} + ijj e_{i+j-3} \ = $\\

& $=$ & $\big(i e_{i-1} \big) \diamond e_j + e_i \diamond \big(j e_{j-1}\big) \ = \
\varphi(e_i) \diamond e_j + e_i \diamond \varphi( e_j).$
\end{longtable}\noindent 
  Then replacing $D$ by $D+\alpha_0\varphi-\beta_0\phi$ we can suppose that 
$D(e_1)\ = \ D(e_2) = \ D(e_3) \ = \ 0.$

Furthermore, using the induction method, we  prove that $ D(e_{k})=0$ for $k\ge 2.$ To do this, let us consider \begin{center}
$D(e_{k+1}) \ = \ \frac {1}{3k} D(e_{3}\diamond e_{k})\ = \ \frac {1}{3k}\big(D(e_{3})\diamond e_{k}+e_ {3}\diamond D(e_{k})\big) \ = \ 0.$
\end{center} It follows that each derivation of ${\rm IDD}(K^0_n,1,1)$ is a linear combination of $\varphi$ and $\phi$, which completes the proof of the statement.
\end{proof}

\begin{corollary}
$\mathfrak{Der}\big({\rm IDD}(K^0_\infty,1,1)\big) = \big\langle \varphi, \phi \big\rangle,$ where $\varphi(e_i)= (i-2) e_i$ and $\phi(e_i) = i e_{i-1}.$
\end{corollary}

\begin{theorem}
If $2\le n < \infty,$ then $\mathfrak{Der}\big({\rm IDD}(K^1_n,1,1)\big) = \big\langle \varphi, \phi \big\rangle,$ where \begin{center} $\varphi(e_i)= (i-2) e_i$ and $\phi(e_i) = (1-\delta_{1,i}) i e_{i-1}.$
\end{center}
\end{theorem}

\begin{proof}
Let $D \in \mathfrak{Der}\big({\rm IDD}(K^1_n,1,1)\big),$ then we can say that  \begin{center}
$D(e_{1})=\sum\limits_{i=1}^{n}\alpha_{i}e_{i},$ \ $D(e_{2})=\sum\limits_{i=1}^{n}\beta_{i}e_{i},$ \ and $D(e_{3})=\sum\limits_{i=1}^{n}\gamma_{i}e_{i}.$
\end{center} Firstly, 
 
\begin{longtable}{lcl}
$0$ & $=$ & $D(e_{1}\diamond e_{1}) \ = \ \left(\sum\limits_{i=1}^{n}\alpha_{i}e_{i}\right)\diamond e_{1}+e_{1}\diamond\left(\sum\limits_{i=1}^{n}\alpha_{i}e_{i}\right) \ =$\\

& $=$ & $\sum\limits_{i=2}^{n}i\alpha_{i}e_{i-1}+\sum\limits_{i=2}^{n}i\alpha_{i}e_{i-1} \ = \ 2\sum\limits_{i=2}^{n}i\alpha_{i}e_{i-1};$
\end{longtable}
\noindent that immediately gives $D(e_{1}) \ = \ \alpha_{1}e_{1}.$ Secondly, \begin{center}
$D(e_{1}) \ = \ \frac 12 D(e_{1}\diamond  e_{2}) \ = \ \frac 12 \left(\left(\alpha_{1}e_{1}\right)\diamond e_{2}+e_{1}\diamond\left(\sum\limits_{i=1}^{n}\beta_{i}e_{i}\right) \right) \ = \ \alpha_{1}e_{1}+\sum\limits_{i=2}^{n}\frac i2 \beta_{i}e_{i-1};$
\end{center} that immediately gives $D(e_{2}) \ = \ \beta_{1}e_{1}.$ Thirdly,   \begin{longtable}{lcl}
$D(e_{2})$ & $=$ & $\frac 13 D(e_{1}\diamond e_{3}) \ = \ \frac 13  \left(\left(\alpha_{1}e_{1}\right)\diamond e_{3}+e_{1}\diamond\left(\sum\limits_{i=1}^{n}\gamma_{i}e_{i}\right) \right) \ = $\\

& $=$ & $\alpha_{1}e_{2}+\sum\limits_{i=2}^{n}\frac i3 \gamma_{i}e_{i-1}\ = \ 
\frac 23 \gamma_{2}e_{1}+(\alpha_{1}+\gamma_{3})e_{2}+\sum\limits_{i=4}^{n}\frac i3 \gamma_{i}e_{i-1};$
\end{longtable} \noindent  that immediately gives $D(e_{3})\ = \ \gamma_{1}e_{1}+\frac{3}{2}\beta_{1}e_{2}-\alpha_{1}e_{3}.$ On the other hand, we have  \begin{longtable}{lcl}
$D(e_{3})$ & $=$ & $ \frac 16 D(e_{2} \diamond e_{3}) \ = \ \frac 16 \left(\left(\beta_{1}e_{1}\right)\diamond e_{3}+e_{2}\diamond\left(\gamma_{1}e_{1}+\frac{3}{2}\beta_{1}e_{2}-\alpha_{1}e_{3}\right) \right) \ = $\\

& $=$ & $ \frac 12 \beta_{1}e_{2}+\frac 13 \gamma_{1}e_{1}+\beta_{1}e_{2}-\alpha_{1}e_{3}\ = \ \frac 13 \gamma_{1}e_{1}+\frac 32 \beta_{1}e_{2}-\alpha_{1}e_{3};$
\end{longtable} \noindent  i.e. $\gamma_{1}=0.$ It is easy to see that linear mappings $\varphi$ and $\phi$ defined by 
\begin{center}
$\varphi(e_i)= (i-2) e_i$ and $\phi(e_i) = (1-\delta_{1,i})ie_{i-1}$  
\end{center} are derivations of ${\rm IDD}(K^1_n,1,1):$  
\begin{longtable}{lcl}
$\varphi(e_i \diamond e_j)$ & $=$ & $ij(i+j-4) e_{i+j-2} \ = \ ij(i-2) e_{i+j-2} + ij(j-2) e_{i+j-2} \ = $\\

& $=$ & $\big((i-2) e_i \big) \diamond e_j + e_i \diamond \big((j-2)e_j\big) \ = \
\varphi(e_i) \diamond e_j + e_i \diamond  \varphi( e_j);$ \\

$\phi(e_i \diamond e_j)$ & $=$ & $ij(1-\delta_{1,i+j-2})(i+j-2) e_{i+j-3} \ =$\\

& $=$ & $(1-\delta_{1,i})i (i-1)j e_{i+j-3} + (1-\delta_{1,j})j i(j-1) e_{i+j-3} \ = $\\

& $=$ & $\big((1-\delta_{1,i})i e_{i-1} \big) \diamond e_j + e_i \diamond \big((1-\delta_{1,j})je_{j-1}\big) \ = \ \phi(e_i) \diamond e_j + e_i \diamond  \phi( e_j).$
\end{longtable}
\noindent 
We can replace $D$ by $D+\alpha_1\varphi - \frac{\beta_1}2 \phi$
and suppose that $D(e_1) \ = \ D(e_2) \ = \ D(e_3)\ = \ 0.$

Furthermore, using the induction method, we prove that $D(e_{k})=0$ for $k\ge 2.$ To do this, let us consider \begin{center}
$D(e_{k+1}) \ = \ \frac 1{3k} D(e_{k}\diamond e_{3}) \ = \ \frac 1{3k}\big(D(e_{k})\diamond e_{3}+e_ {k}\diamond D(e_{3})\big) \ = \ 0.$
\end{center} It follows that each derivation of ${\rm IDD}(K^1_n,2,0)$ is a linear combination of $\varphi$, that completes the proof of the statement.
\end{proof}

\begin{corollary}
$\mathfrak{Der}\big({\rm IDD}(K^1_\infty,1,1)\big) = \big\langle \varphi, \phi \big\rangle,$ where \begin{center} $\varphi(e_i)= (i-2) e_i$ and $\phi(e_i) = (1-\delta_{1,i}) i e_{i-1}.$
\end{center}
\end{corollary}

\subsection{Derivations of \texorpdfstring{$\mathrm{IDD}(K^\times_\mathfrak{n},1,-1)$}{IDD(K\^x\_n,1,-1)}}\label{der1-1}

\begin{theorem}
If $1\le n < \infty,$ then $\mathfrak{Der}\big({\rm IDD}(K^0_n,1,-1)\big) = \big\langle\varphi \big\rangle,$ where $\varphi(e_i)=ie_i.$
\end{theorem}

\begin{proof}
Let $D \in \mathfrak{Der}\big({\rm IDD}(K^0_n,1,-1)\big),$ then $D(e_{0})=\sum\limits_{i=0}^{n}\alpha_{i}e_{i}$ and $D(e_{1})=\sum\limits_{i=0}^{n}\beta_{i}e_{i}.$ Firstly, \begin{center}
$0 \ = \ D(e_{0}\diamond e_{0}) \ = \ \left(\sum\limits_{i=0}^{n}\alpha_{i}e_{i}\right)\diamond e_{0}+e_{0}\diamond\left(\sum\limits_{i=0}^{n}\alpha_{i}e_{i}\right) \ = \ \sum\limits_{i=0}^{n}i\alpha_{i}e_{i},$
\end{center} that immediately gives $\alpha_{k}=0$ for $1\le k\le n,$ i.e., $D(e_{0})=\alpha_{0}e_{0}.$ Secondly,  
\begin{longtable}{lcl}
$D(e_{1})$ & $=$ & $D(e_{1}\diamond e_{0}) \ = \ \left(\sum\limits_{i=0}^{n}\beta_{i}e_{i}\right)\diamond e_{0}+e_{1}\diamond\left(\alpha_{0}e_{0}\right) \ = \ \sum\limits_{i=0}^{n}i\beta_{i}e_{i}+\alpha_{0}e_{1} \ =$ \\

& $=$ & $\left(\beta_{1}+\alpha_{0}\right)e_{1}+\sum\limits_{i=2}^{n}i\beta_{i}e_{i},$
\end{longtable}
\noindent  that immediately gives $ \beta_{0}=0,$ $\alpha_{0}=0,$ and $\beta_{k}=0$ for $2\le k\le n,$ i.e., \begin{center}
$D(e_{0})=0$ and $D(e_{1})=\beta_{1}e_{1}.$ 
\end{center}

It is easy to see that $\varphi(e_i)=  i e_{i} $ is a derivation of ${\rm IDD}(K^0_n,1,-1):$  
\begin{longtable}{lcl}
$\varphi(e_i \diamond e_j)$ & $=$ & $\frac i{j+1} (i+j) e_{i+j} \ = \ \frac {i^2}{j+1}  e_{i+j} + \frac{ij}{j+1} e_{i+j} \ =$\\

& $=$ & $\big( i e_i \big) \diamond e_j + e_i \diamond \big(j e_j\big) \ = \
\varphi(e_i) \diamond e_j + e_i \diamond \varphi( e_j).$
\end{longtable}
 \noindent  Then replacing $D$ by $D-\beta_1 \varphi$ we can suppose that $D(e_0) \ = \ D(e_1) \ = \ 0.$

Furthermore, using the induction method, we prove that $D(e_{k})=0$ for $k\ge 2.$ To do this, let us consider \begin{center}
$D(e_{k+1}) \ = \ \left(k+1\right)D(e_{1}\diamond e_{k}) \ = \ \left(k+1\right)\big(D(e_{1})\diamond e_{k}+e_ {1}\diamond D(e_{k})\big) \ = \ 0.$
\end{center} It follows that each derivation of ${\rm IDD}(K^0_n,1,-1)$ is a linear combination of $\varphi$, that completes the proof of the statement.
\end{proof}

\begin{corollary}
$\mathfrak{Der}\big({\rm IDD}(K^0_\infty,1,-1)\big) = \big\langle\varphi  \big\rangle,$ where $\varphi(e_i) \ = i e_i.$
\end{corollary}

\begin{proposition}
$\mathfrak{Der}\big({\rm IDD}(K^1_2,1,-1)\big) = \big\langle \varphi_1, \varphi_2 \big\rangle,$ where  
\begin{longtable}{|c|c|}
\hline
$\varphi_{2}$ & $\varphi_{2}(e_{1})=e_{2}$ \\
\hline
$\varphi_{1}$ & $\varphi_{1}(e_{1})=e_{1},$ $\varphi_{1}(e_{2})=2e_{2}$ \\
\hline
\end{longtable}

\end{proposition}

\begin{theorem}
If $3\le n < \infty,$ then $\mathfrak{Der}\big({\rm IDD}(K^1_n,1,-1)\big) = \big\langle\varphi,\  \phi_1,\  \phi_2 \big\rangle,$ where  
\begin{longtable}{|l|c|}
\hline 
$\varphi$&$\varphi(e_i)  =   i e_i$  \\ 
\hline 
$\phi_1$&$\phi_1(e_1)  =  n e_{n-1}, \ \phi_1(e_2)  =  (n^{2}-n+2) e_{n}$ \\
\hline
$\phi_2$&$\phi_2(e_1)  =  e_n$\\
\hline
\end{longtable}
\end{theorem}

\begin{proof}
Let $D \in \mathfrak{Der}\big({\rm IDD}(K^1_n,1,-1)\big),$ then we can say that  $D(e_{1})=\sum\limits_{i=1}^{n}\alpha_{i}e_{i}.$ Firstly, 
\begin{longtable}{lcl}
$D(e_{2})$ & $=$ & $2D(e_{1}\diamond e_{1})\ = \ 2 \left(\left(\sum\limits_{i=1}^{n}\alpha_{i}e_{i}\right)\diamond e_{1}+e_{1}\diamond\left(\sum\limits_{i=1}^{n}\alpha_{i}e_{i}\right) \right) \ =$ \\

& $=$ & $\sum\limits_{i=1}^{n-1}{i} \alpha_{i}e_{i+1}+\sum\limits_{i=1}^{n-1}\frac{2}{i+1}\alpha_{i}e_{i+1}\ = \
\sum\limits_{i=1}^{n-1}\frac{i^{2}+i+2}{i+1}\alpha_{i}e_{i+1};$\\ 

$D(e_{3})$ & $=$ & $3 D(e_{1}\diamond e_{2})\ = \ 
3 \left(\left(\sum\limits_{i=1}^{n}\alpha_{i}e_{i}\right)\diamond e_{2}+e_{1}\diamond\left(\sum\limits_{i=1}^{n-1}\frac{i^{2}+i+2}{i+1}\alpha_{i}e_{i+1}\right) \right) \ = $ \\

& $=$ & $\sum\limits_{i=1}^{n-2}{i}\alpha_{i}e_{i+2}+\sum\limits_{i=1}^{n-2}\frac{3(i^{2}+i+2)}{(i+1)(i+2)}\alpha_{i}e_{i+2}\ =\ \sum\limits_{i=1}^{n-2}\frac{i^{3}+6i^{2}+5i+6}{ (i+1)(i+2)}\alpha_{i}e_{i+2}.$
\end{longtable}\noindent On the other hand:  
\begin{longtable}{lcl}
$D(e_{3})$ & $=$ & $ D(e_{2}\diamond e_{1}) \ = \ \left(\sum\limits_{i=1}^{n-1}\frac{i^{2}+i+2}{i+1}\alpha_{i}e_{i+1}\right)\diamond e_{1}+e_{2}\diamond\left(\sum\limits_{i=1}^{n}\alpha_{i}e_{i}\right)=$ \\

& $=$ & $\sum\limits_{i=1}^{n-2}\frac{i^{2}+i+2}{2}\alpha_{i}e_{i+2}+\sum\limits_{i=1}^{n-2}\frac{2}{i+1}\alpha_{i}e_{i+2}\ = \ 
\sum\limits_{i=1}^{n-2}\frac{i^{3}+2i^{2}+3i+6}{2(i+1)}\alpha_{i}e_{i+2}.$
\end{longtable} \noindent  Comparing these two presentations of $D(e_{3}),$ we have that \begin{center}
$i\big(i^{3}+2i^{2}-5i+2\big) \alpha_{i}\ =\ 0$ for $1 \le i\le n-2.$
\end{center} The roots of the equation $i\big(i^{3}+2i^{2}-5i+2\big)=0$ are $0,\ 1,$ or $\frac{1}{2}\big(-3\pm \sqrt{17}\big).$ Then, we obtain $\alpha_{2}=\alpha_{3}=\ldots=\alpha_{n-2}=0.$ Summarizing, we obtain: 
\begin{center} $D(e_{1}) = \alpha_{1}e_{1}+\alpha_{n-1}e_{n-1}+\alpha_{n}e_{n},$ \ $D(e_{2}) = 2\alpha_{1}e_{2}+\frac{n^{2}-n+2}{n}\alpha_{n-1}e_{n},$ \ and \
$D(e_{3}) = 3\alpha_{1}e_{3}.$ \end{center} It is easy to see that linear mappings $\varphi, \ \phi_1,$ and $\phi_2$ defined by \begin{center}
$\varphi(e_i) \ = \ i e_i,$ \ $\phi_1(e_1) \ = \ n e_{n-1},$ \ $\phi_1(e_2) \ = \ (n^{2}-n+2) e_{n},$ \ $\phi_2(e_1) \ = \ e_n$
\end{center} are derivations of ${\rm IDD}(K^1_n,1,-1):$ 
\begin{longtable}{lcl}
$\varphi(e_i \diamond e_j)$ & $=$ & $\frac{i}{j+1} (i+j) e_{i+j} \ = \ 
i\frac{i}{j+1} e_{i+j} + j \frac i{j+1} e_{i+j} \ = $ \\
& $=$ & $ \big(i e_i\big) \diamond e_j + e_i \diamond \big(j e_j\big) \ = \
\varphi(e_i) \diamond e_j + e_i \diamond  \varphi( e_j);$ \\

$\phi_1(e_1 \diamond e_1)$ & $=$ & $\frac1{2}\phi_1(e_2) \ = \ \frac{n^{2}-n+2}{2}e_{n} \ = \ \frac{n^{2}-n}{2n}e_n + \frac{n}{n} e_n \ = $ \\
& $=$ & $\left(ne_{n-1}\right) \diamond e_1+e_1 \diamond \left(ne_{n-1}\right) \ = \ 
\phi_1(e_1) \diamond e_1+e_1 \diamond \phi_1(e_1).$
\end{longtable}
\noindent 
We can replace $D$ by $D-\alpha_1\varphi - \frac{1}{n}\alpha_{n-1}\phi_1- \alpha_n \phi_2$ and suppose that $D(e_1) \ = \ D(e_2) \ = \ 0.$

Furthermore, using the induction method, we can prove that $D(e_{k})=0,$ for $k\ge 3.$ To do this, let us consider \begin{center}
$D(e_{k+1}) \ = \ \frac{2}{k}D(e_{k}\diamond e_{1}) \ = \ \frac{2}{k} \big(D(e_{k})\diamond e_{1}+e_ {k}\diamond D(e_{1})\big) \ = \ 0.$
\end{center}

It follows that each derivation of ${\rm IDD}(K^1_n,1,-1)$ is a linear sum of  $\varphi,$ $\phi_1,$ and $\phi_2.$ That completes the proof of the statement.
\end{proof}

\begin{corollary}
$\mathfrak{Der}\big({\rm IDD}(K^1_\infty,1,-1)\big) = \big\langle\varphi  \big\rangle,$ where $\varphi(e_i) \ = i e_i.$
\end{corollary}

\subsection{Derivations of \texorpdfstring{$\mathrm{IDD}(K^\times_\mathfrak{n},-1,-1)$}{IDD(K\^x\_n,-1,-1)}}\label{der1-1-1}

\begin{proposition}
$\mathfrak{Der}\big({\rm IDD}(K^0_2,-1,-1)\big) = \big\langle \varphi_i, \phi_j \big\rangle_{ 0\le i \le 2}^{1\le j \le 2},$ where  
\begin{longtable}{|c|c|l|}
\hline 
$\varphi_i$ & $\varphi_i(e_{0}) \ = \ e_{i}, \ \varphi_i(e_{2}) \ = \ 2 {\delta}_{0,i}e_{2} $ & $0\leq i\leq 2$ \\
\hline
$\phi_j$ & $\phi_j(e_{1}) \ = \  e_{j}$ & $1\leq j\leq 2$ \\
\hline
\end{longtable}
 
\end{proposition}

\begin{proposition}
$\mathfrak{Der}\big({\rm IDD}(K^0_3,-1,-1)\big) = \big\langle \varphi_i, \phi_j \big\rangle_{ 1\le i \le 3}^{1\le j \le 3},$ where  
\begin{longtable}{|c|c|l|}
\hline
$\varphi_i$ & $\varphi_i (e_0)=e_i,$ \ $\varphi_1 (e_2)=e_3$ & $1 \le i \le 3$ \\
\hline
$\phi_j$ & $\phi_j (e_1)=e_j,$ \ $\phi_1 (e_3)=e_3$ & $1 \le j \le 3$ \\
\hline
\end{longtable}
 
\end{proposition}

\begin{proposition}
$\mathfrak{Der}\big({\rm IDD}(K^0_4,-1,-1)\big) = \big\langle \varphi_i, \phi_j \big\rangle_{ 0\le i \le 4}^{2\le j \le 4},$ where 
\begin{longtable}{|c|c|l|}
\hline
$\varphi_4$ & $\varphi_4 (e_0)=e_4$ \\
\hline
$\varphi_3$ & $\varphi_3 (e_0)=e_3$ \\
\hline
$\varphi_2$ & $\varphi_2 (e_0)= 3 e_2,$ $\varphi_2 (e_2)= 2 e_4$ \\
\hline
$\varphi_1$ & $\varphi_1 (e_0)= 2 e_1,$ $\varphi_1 (e_2)= 2 e_3,$ $\varphi_1 (e_3)=e_4,$ \\
\hline
$\varphi_0$ & $\varphi_0 (e_0)= 2 e_0,$ $\varphi_0 (e_1)= 3 e_1,$ $\varphi_0 (e_2)= 4 e_2,$ $\varphi_0 (e_3)= 5 e_3,$ $\varphi_0 (e_4)= 6 e_4$ \\
\hline
$\phi_4$ & $\phi_4 (e_1)=e_4$ \\
\hline
$\phi_3$ & $\phi_3 (e_1)=e_3$ \\
\hline
$\phi_2$ & $\phi_2 (e_1)= 3 e_2,$ $\phi_2 (e_3)= 2 e_4$ \\
\hline
\end{longtable}
 
\end{proposition}

\begin{theorem}

If $5\le n < \infty,$ then $\mathfrak{Der}\big({\rm IDD}(K^0_n,-1,-1)\big) = \big\langle \varphi, \phi_i, \psi_j\big\rangle_{ 0\le i \le 4}^{0\le j \le 2},$ where  
\begin{longtable}{|c|c|}
\hline 
$\varphi$ & $\varphi(e_{i}) \ = \ \left(i+2\right) e_{i}, \ 0\leq i\leq n$\\
\hline
$\phi_4$ &
$\phi_4(e_0) = 4(n-1)(n-3) e_{n-4}, \ 
\phi_4(e_1) = (n-2)(n+5) e_{n-3},$ \\
& $
\phi_4(e_2) = 8(n-1) e_{n-2}, \
\phi_4(e_3) = 6(n+1) e_{n-1}, \
\phi_4(e_4) = 4(n+5) e_{n}$ \\
\hline
$\phi_3$ & $
\phi_3(e_0) = (n-2) e_{n-3}, \  
\phi_3(e_2) = 2 e_{n-1},\  
\phi_3(e_3) = e_{n}$ \\
\hline
$\phi_2$ & $\phi_2(e_0) = (n-1) e_{n-2},\  
\phi_2(e_2) = 2 e_{n}$ \\
\hline
$\phi_1$ & $\phi_1(e_0) = e_{n-1}$ \\
\hline
$\phi_0$ & $\phi_0(e_0) = e_{n}$ \\
\hline
$\psi_2$ & $\psi_2(e_1) = (n-1) e_{n-2},
\psi_2(e_3) = e_{n}$ \\
\hline
$\psi_1$ & $\psi_1(e_1) = e_{n-1}$ \\
\hline
$\psi_0$ & $\psi_0(e_1) = e_{n}$ \\
\hline
\end{longtable}

\end{theorem}

\begin{proof}
Let $D \in \mathfrak{Der}\big({\rm IDD}(K^0_n,-1,-1)\big),$ then $D(e_{0})=\sum\limits_{i=0}^{n}\alpha_{i}e_{i}$ and $ D(e_{1})=\sum\limits_{i=0}^{n}\beta_{i}e_{i}.$ Firstly,  
\begin{longtable}{lcl}
$D(e_{2})$ & $=$ & $D(e_{0}\diamond e_{0}) \ = \ \left(\sum\limits_{i=0}^{n}\alpha_{i}e_{i}\right)\diamond e_{0}+e_{0}\diamond\left(\sum\limits_{i=0}^{n}\alpha_{i}e_{i}\right) \ =$ \\

& $=$ & $\sum\limits_{i=0}^{n-2}\frac{1}{i+1}\alpha_{i}e_{i+2}+\sum\limits_{i=0}^{n-2}\frac{1}{i+1}\alpha_{i}e_{i+2} \ = \ \sum\limits_{i=0}^{n-2}\frac{2}{i+1}\alpha_{i}e_{i+2};$ \\

$D(e_{3})$ & $=$ & $2D(e_{0}\diamond e_{1}) \ = \ 2\left(\left(\sum\limits_{i=0}^{n}\alpha_{i}e_{i}\right)\diamond e_{1}+e_{0}\diamond\left(\sum\limits_{i=0}^{n}\beta_{i}e_{i}\right)\right) \ =$ \\

& $=$ & $\sum\limits_{i=0}^{n-3}\frac{1}{i+1}\alpha_{i}e_{i+3}+\sum\limits_{i=0}^{n-2}\frac{2}{i+1}\beta_{i}e_{i+2} \ = \ 2\beta_{0}e_{2}+\sum\limits_{i=0}^{n-3}\left(\frac{1}{i+1}\alpha_{i}+\frac{2}{i+2}\beta_{i+1}\right)e_{i+3}.$ \\
\end{longtable}
\noindent  Secondly, we find two expressions for $D(e_4):$

\begin{longtable}{lcl}
$D(e_{4})$ & $=$ & $ 3D(e_{0}\diamond e_{2}) \ = \ 3\left(\left(\sum\limits_{i=0}^{n}\alpha_{i}e_{i}\right)\diamond e_{2}+e_{0}\diamond\left(\sum\limits_{i=0}^{n-2}\frac{2}{i+1}\alpha_{i}e_{i+2}\right)\right) \ =$ \\

& $=$ & $\sum\limits_{i=0}^{n-4}\frac{1}{i+1}\alpha_{i}e_{i+4}+\sum\limits_{i=0}^{n-4}\frac{6}{(i+1)(i+3)}\alpha_{i}e_{i+4} \ = \ \sum\limits_{i=0}^{n-4}\frac{i+9}{(i+1)(i+3)}\alpha_{i}e_{i+4};$\\

$D(e_{4})$ & $=$ & $4D(e_{1}\diamond e_{1}) \ = \ 4\left(\left(\sum\limits_{i=0}^{n}\beta_{i}e_{i}\right)\diamond e_{1}+e_{1}\diamond\left(\sum\limits_{i=0}^{n}\beta_{i}e_{i}\right)\right) \ =$ \\

& $=$ & $\sum\limits_{i=0}^{n-3}\frac{2}{i+1}\beta_{i}e_{i+3}+\sum\limits_{i=0}^{n-3}\frac{2}{i+1}\beta_{i}e_{i+3} \ = \ 4\beta_{0}e_{3}+\sum\limits_{i=0}^{n-4}\frac{4}{i+2}\beta_{i+1}e_{i+4};$ \\
\end{longtable}
 \noindent  that immediately gives $\beta_{0}=0$ and $\beta_{k}=\frac{(k+1)(k+8)}{4k(k+2)}\alpha_{k-1}$ for $ 1\le k\le n-3,$ i.e., 
\begin{longtable}{lcl}
$D(e_{1})$ & $=$ & $\sum\limits_{i=1}^{n-3}\frac{(i+1)(i+8)}{4i(i+2)}\alpha_{i-1}e_{i}+\beta_{n-2}e_{n-2}+\beta_{n-1}e_{n-1}+\beta_{n}e_{n};$ \\

$D(e_{3})$ & $=$ & $\sum\limits_{i=0}^{n-4}\frac{3(i+5)}{2(i+1)(i+3)}\alpha_{i}e_{i+3}+(\frac{1}{n-2}\alpha_{n-3}+\frac{2}{n-1}\beta_{n-2})e_{n}.$
\end{longtable}
\noindent  Thirdly, we find different expressions for $D(e_5):$  

\begin{longtable}{lcl}
$D(e_{5})$ & $=$ & $4D(e_{0}\diamond e_{3}) \ = \ 4\Big(\left(\sum\limits_{i=0}^{n}\alpha_{i}e_{i}\right)\diamond e_{3}+e_{0}\diamond\left(\sum\limits_{i=0}^{n-4}\frac{3(i+5)}{2(i+1)(i+3)}\alpha_{i}e_{i+3}\right) \Big) \ =$ \\

& $=$ & $\sum\limits_{i=0}^{n-5}\frac{1}{i+1}\alpha_{i}e_{i+5}+\sum\limits_{i=0}^{n-5}\frac{6(i+5)}{(i+1)(i+3)(i+4)}\alpha_{i}e_{i+5}$\\
&$=$ &$\sum\limits_{i=0}^{n-5}\frac{(i+6)(i+7)}{(i+1)(i+3)(i+4)}\alpha_{i}e_{i+5};$ \\[1ex] \\

$D(e_{5})$ & $=$ & $6D(e_{1}\diamond e_{2}) \ = \ 6\big(D(e_{1})\diamond e_{2}+e_{1}\diamond D(e_{2})\big) \ =$\\

& $=$ & $6\Big(\left(\sum\limits_{i=1}^{n-3}\frac{(i+1)(i+8)}{4i(i+2)}\alpha_{i-1}e_{i}
+\beta_{n-2}e_{n-2}+\beta_{n-1}e_{n-1}+\beta_{n}e_{n}\right)\diamond e_{2}+$\\

& & $e_{1}\diamond\left(\sum\limits_{i=0}^{n-2}\frac{2}{i+1}\alpha_{i}e_{i+2}\right) \Big) \ =$\\

& $=$ & $\sum\limits_{i=1}^{n-4}\frac{i+8}{2i(i+2)}\alpha_{i-1}e_{i+4}
+\sum\limits_{i=0}^{n-5}\frac{6}{(i+1)(i+3)}\alpha_{i}e_{i+5} \ =$\\

& $=$ & $\sum\limits_{i=0}^{n-5}\frac{i+21}{2(i+1)(i+3)}\alpha_{i}e_{i+5}.$ \\
\end{longtable}
\noindent  Comparing the present expressions, we obtain $\alpha_{k}=0$ for $1\le k\le n-5;$ i.e.,  
\begin{longtable}{lcl}
$D(e_{0})$ & $=$ & $\alpha_{0}e_{0}+\alpha_{n-4}e_{n-4}+\alpha_{n-3}e_{n-3}+\alpha_{n-2}e_{n-2}+\alpha_{n-1}e_{n-1}+\alpha_{n}e_{n};$ \\

$D(e_{1})$ & $=$ & $\frac 32\alpha_{0}e_{1}+\frac{(n-2)(n+5)}{4(n-1)(n-3)}\alpha_{n-4}e_{n-3}+\beta_{n-2}e_{n-2}+\beta_{n-1}e_{n-1}+\beta_{n}e_{n};$ \\

$D(e_{2})$ & $=$ & $2\alpha_{0}e_{2}+\frac{2}{n-3}\alpha_{n-4}e_{n-2}+\frac{2}{n-2}\alpha_{n-3}e_{n-1}+\frac{2}{n-1}\alpha_{n-2}e_{n};$\\

$D(e_{3})$ & $=$ & $\frac{5}{2}\alpha_{0}e_{3}+\frac{3(n+1)}{2(n-1)(n-3)}\alpha_{n-4}e_{n-1}+\left(\frac{1}{n-2}\alpha_{n-3}+\frac{1}{n-1}\beta_{n-2}\right)e_{n};$\\

$D(e_{4})$ & $=$ & $3\alpha_{0}e_{4}+\frac{n+5}{(n-1)(n-3)}\alpha_{n-4}e_{n};$ \\

$D(e_{5})$ & $=$ & $\frac{7}{2}\alpha_{0}e_{5},$ and $D(e_{6})$ \ = \ $4\alpha_{0}e_{6}.$ \end{longtable} 

Furthermore, using the induction method, we can prove that $D(e_{k})=\frac{k+2}{2}\alpha_{0}e_{k},$ for $5 \le k\le n.$ To do this, let us consider  
\begin{longtable}{lcl}
$D(e_{k+2})$ & $=$ & $(k+1)D(e_{k}\diamond e_{0}) \ = \ 
(k+1) \big(D(e_{k})\diamond e_{0}+e_{k}\diamond D(e_{0}) \big)\ = $\\

& $=$ & $(k+1)\left(\left(\frac{k+2}{2}\alpha_{0}e_{k}\right)\diamond e_{0}
+e_{k}\diamond\left(\alpha_{0}e_{0}+\sum\limits_{i=n-4}^n\alpha_{i}e_{i}\right)\right) \ = \  \frac{k+4}{2}\alpha_{0}e_{k+2}.$
\end{longtable}
\noindent  Combining all the obtained results, we see that $D$ is a linear combination of mappings
$\left\{ \varphi, \phi_i, \psi_j\right\}_{ 0\le i \le 4}^{0\le j \le 2}.$
\end{proof}

\begin{corollary}
$\mathfrak{Der}\big({\rm IDD}(K^0_\infty,-1,-1)\big) = \big\langle \varphi\big\rangle,$ where $\varphi(e_{i}) = \left(i+2\right) e_{i}.$
\end{corollary}

\begin{proposition}
$\mathfrak{Der}\big({\rm IDD}(K^1_4,-1,-1)\big) = \big\langle \varphi_i, \phi_j, \psi_j\big\rangle_{1\le i\le 4}^{2\le j\le 4},$ where  
\begin{longtable}{|c|c|l|}
\hline
$\varphi_i$ & $\varphi_{i}(e_k)=\delta_{1,k}e_{i}+2\delta_{1,i}\delta_{4,k}e_{4}$ & $1\le i\le 4$ \\
\hline
$\phi_j$ & $\phi_{j}(e_{2})=e_{j}$ & $2\le j\le 4$ \\
\hline
$\psi_j$ & $\psi_{j}(e_{3})=e_{j}$ & $2\le j\le 4$ \\
\hline
\end{longtable}
 
\end{proposition}

\begin{proposition}
$\mathfrak{Der}\big({\rm IDD}(K^1_5,-1,-1)\big) =  \big\langle \varphi_i, \phi_j, \psi_k\big\rangle_{1\le i\le 5; \ 2\le j\le 5}^{3\le k\le 5},$ where  
\begin{longtable}{|c|c|}
\hline
$\varphi_{5}$ & $\varphi_{5}(e_{1})=e_{5}$ \\
\hline
$\varphi_{4}$ & $\varphi_{4}(e_{1})=e_{4}$ \\
\hline
$\varphi_{3}$ & $\varphi_{3}(e_{1})=e_{3}$ \\
\hline
$\varphi_{2}$ & $\varphi_{2}(e_{1})= 3 e_{2},$ \ $\varphi_{2}(e_{4})= 4 e_{5}$ \\
\hline
$\varphi_{1}$ & $\varphi_{1}(e_{1})=e_{1},$ \ $\varphi_{1}(e_{4})=2e_{4},$ \ $\varphi_{1}(e_{5})=e_{5}$ \\
\hline
$\phi_{5}$ & $\phi_{5}(e_{2})=e_{5}$ \\
\hline
$\phi_{4}$ & $\phi_{4}(e_{2})=e_{4}$ \\
\hline
$\phi_{3}$ & $\phi_{3}(e_{2})=e_{3}$ \\
\hline
$\phi_{2}$ & $\phi_{2}(e_{2})=e_{2},$ \ $\phi_{2}(e_{5})=e_{5},$ \\
\hline
$\psi_{5}$ & $\psi_{5}(e_{3})=e_{5}$ \\
\hline
$\psi_{4}$ & $\psi_{4}(e_{3})=e_{4}$ \\
\hline
$\psi_{3}$ & $\psi_{3}(e_{3})=e_{3}$ \\
\hline
\end{longtable}
\end{proposition}

\begin{proposition}
$\mathfrak{Der}\big({\rm IDD}(K^1_6,-1,-1)\big) =  \big\langle \varphi_i, \phi_j, \psi_k\big\rangle_{1\le i\le 6; \ 2\le j\le 6}^{4\le k\le 6},$ where  
\begin{longtable}{|c|c|}
\hline
$\varphi_{6}$ & $\varphi_{6}(e_{1})=e_{6}$ \\
\hline
$\varphi_{5}$ & $\varphi_{5}(e_{1})=e_{5}$ \\
\hline
$\varphi_{4}$ & $\varphi_{4}(e_{1})=e_{4}$ \\
\hline
$\varphi_{3}$ & $\varphi_{3}(e_{1})=e_{3},$ \ $\varphi_{3}(e_{4})=e_{6}$ \\
\hline
$\varphi_{2}$ & $\varphi_{2}(e_{1})= 3 e_{2},$ \ $\varphi_{2}(e_{4})= 4 e_{5},$ \ $\varphi_{2}(e_{5})= 2 e_{6}$ \\
\hline
$\varphi_{1}$ & $\varphi_{1}(e_{1})=e_{1},$ \ $\varphi_{1}(e_{3})=-e_{3},$ \ $\varphi_{1}(e_{4})=2e_{4},$ \ $\varphi_{1}(e_{5})=e_{5}$ \\
\hline
$\phi_{6}$ & $\phi_{6}(e_{2})=e_{6}$ \\
\hline
$\phi_{5}$ & $\phi_{5}(e_{2})=e_{5}$ \\
\hline
$\phi_{4}$ & $\phi_{4}(e_{2})=e_{4}$ \\
\hline
$\phi_{3}$ & $\phi_{3}(e_{2})= 4 e_{3},$ \ $\phi_{3}(e_{5})= 3 e_{6}$ \\
\hline
$\phi_{2}$ & $\phi_{2}(e_{2})=e_{2},$ \ $\phi_{2}(e_{3})=2e_{3},$ \ $\phi_{2}(e_{5})=e_{5},$ \ $\phi_{2}(e_{6})=2e_{6}$ \\
\hline
$\psi_{6}$ & $\psi_{6}(e_{3})=e_{6}$ \\
\hline
$\psi_{5}$ & $\psi_{5}(e_{3})=e_{5}$ \\
\hline
$\psi_{4}$ & $\psi_{4}(e_{3})=e_{4}$ \\
\hline
    \end{longtable}
\end{proposition}

\begin{proposition}
$\mathfrak{Der}\big({\rm IDD}(K^1_7,-1,-1)\big) = \big\langle \varphi_i, \phi_j, \psi_k\big\rangle_{1\le i\le 7; \ 3\le j\le 7}^{5\le k\le 7},$ where  
\begin{longtable}{|c|c|}
\hline
$\varphi_{7}$ & $\varphi_{7}(e_{1})=e_{7}$ \\
\hline
$\varphi_{6}$ & $\varphi_{6}(e_{1})=e_{6}$ \\
\hline
$\varphi_{5}$ & $\varphi_{5}(e_{1})=e_{5}$ \\
\hline
$\varphi_{4}$ & $\varphi_{4}(e_{1})= 5 e_{4},$ \ $\varphi_{4}(e_{4})= 4 e_{7}$ \\
\hline
$\varphi_{3}$ & $\varphi_{3}(e_{1})= 2 e_{3},$ \ $\varphi_{3}(e_{4})= 2 e_{6},$ \ $\varphi_{3}(e_{5})=e_{7}$ \\
\hline
$\varphi_{2}$ & $\varphi_{2}(e_{1})= 6 e_{2},$ \ $\varphi_{2}(e_{3})= 5 e_{4},$ \ $\varphi_{2}(e_{4})= -8 e_{5},$ \ $\varphi_{2}(e_{5})= 4 e_{6}$ \\
\hline
$\varphi_{1}$ & $\varphi_{1}(e_{1})= 3 e_{1},$ \ $\varphi_{1}(e_{2})= 4 e_{2},$ \ $\varphi_{1}(e_{3})= 5 e_{3},$ \ $\varphi_{1}(e_{4})=2e_{4},$ \\
& $\varphi_{1}(e_{5})= 7 e_{5},$ \ $\varphi_{1}(e_{6})= 8 e_{6},$ \ $\varphi_{1}(e_{7})=3e_{7}$ \\
\hline
$\phi_{7}$ & $\phi_{7}(e_{2})=e_{7}$ \\
\hline
$\phi_{6}$ & $\phi_{6}(e_{2})=e_{6}$ \\
\hline
$\phi_{5}$ & $\phi_{5}(e_{2})=e_{5}$ \\
\hline
$\phi_{4}$ & $\phi_{4}(e_{2})= 5 e_{4},$ \ $\phi_{4}(e_{5})= 3 e_{7}$ \\
\hline
$\phi_{3}$ & $\phi_{3}(e_{2})= 8 e_{3},$ \ $\phi_{3}(e_{3})= 15 e_{4},$ \ $\phi_{3}(e_{5})= 6 e_{6},$ \ $\phi_{3}(e_{6})= 12 e_{7}$ \\
\hline
$\psi_{7}$ & $\psi_{7}(e_{3})=e_{7}$ \\
\hline
$\psi_{6}$ & $\psi_{6}(e_{3})=e_{6}$ \\
\hline
$\psi_{5}$ & $\psi_{5}(e_{3})=e_{5}$ \\
\hline
\end{longtable}
 
\end{proposition}

\begin{theorem}
If $8\le n < \infty,$ then $\mathfrak{Der}\big({\rm IDD}(K^1_n,-1,-1)\big) = \big\langle \varphi, \phi_i, \psi_j, \pi_k\big\rangle_{ 0\le i \le 6; \ 0\le j \le 4}^{ 0\le k \le 2},$ where  
\begin{longtable}{|c|c|}
\hline 
$\varphi$ & $\varphi(e_{i}) \ = \ \left(i+2\right) e_{i}, \ 1\leq i\leq n$ \\
\hline
& $\phi_6(e_1) = 18(n-2)(n-5) e_{n-6}, \ 
\phi_6(e_2) = 8(n-4)(n+3) e_{n-5},$ \\
$\phi_6$ & $\phi_6(e_3) = 3(n-3)(n+18) e_{n-4},$ \ 
$\phi_6(e_4) = 72(n-2) e_{n-3},$\\
& $\phi_6(e_5) = 60n e_{n-2}, \ 
\phi_6(e_6) = 48(n+3) e_{n-1}, \ 
\phi_6(e_7) = 36(n+8) e_{n}$ \\
\hline
$\phi_5$ & $\phi_5(e_1) = 2(n-4) e_{n-5}, \ 
\phi_5(e_3) = (2-n) e_{n-3}, \ 
\phi_5(e_4) = 8 e_{n-2}, \
\phi_5(e_5) = 4 e_{n-1}$ \\
\hline
$\phi_4$ & $\phi_4(e_1) = (n-3) e_{n-4}, \  
\phi_4(e_4) = 4 e_{n-1},\  
\phi_4(e_5) = 2 e_{n}$ \\
\hline
$\phi_3$ & $\phi_3(e_1) = (n-2) e_{n-3},\  
\phi_3(e_4) = 4 e_{n}$ \\
\hline
$\phi_2$ & $\phi_2(e_1) = e_{n-2}$ \\
\hline
$\phi_1$ & $\phi_1(e_1) = e_{n-1}$ \\
\hline
$\phi_0$ & $\phi_0(e_1) = e_{n}$  \\
\hline
$\psi_4$ & $\psi_4(e_2) = 2(n-3) e_{n-4}, \
\psi_4(e_3) = 3(n-2) e_{n-3}, \
\psi_4(e_5) = -6 e_{n-1}, \ 
\psi_4(e_6) = 12 e_{n}$ \\
\hline
$\psi_3$ & $\psi_3(e_2) = (n-2)e_{n-3}, \psi_3(e_5) = 3 e_{n}$  \\
\hline
$\psi_2$ & $\psi_2(e_2) = e_{n-2}$ \\
\hline
$\psi_1$ & $\psi_1(e_2) = e_{n-1}$ \\
\hline
$\psi_0$ & $\psi_0(e_2) = e_{n}$ \\
\hline
$\pi_2$ & $\pi_2(e_3) = e_{n-2}$ \\
\hline
$\pi_1$ & $\pi_1(e_3) = e_{n-1}$ \\
\hline
$\pi_0$ & $\pi_0(e_3) = e_{n}$ \\
\hline
\end{longtable}
 
\end{theorem}

\begin{proof}
Let $D \in \mathfrak{Der}\big({\rm IDD}(K^1_n,-1,-1)\big),$ then we can say that
\begin{center}
$D(e_{1})=\sum\limits_{i=1}^{n}\alpha_{i}e_{i}, \ D(e_{2})=\sum\limits_{i=1}^{n}\beta_{i}e_{i},$ and $ D(e_{3})=\sum\limits_{i=1}^{n}\gamma_{i}e_{i}.$
\end{center} Firstly, we find expressions for $D(e_4),$ $D(e_5),$ and $D(e_6):$ 
\begin{longtable}{lcl}
$D(e_{4})$
&$=$&$ 4D(e_{1}\diamond e_{1}) 
 \ =\  4\left(\left(\sum\limits_{i=1}^{n}\alpha_{i}e_{i}\right)\diamond 
e_{1}+e_{1}\diamond\left(\sum\limits_{i=1}^{n}\alpha_{i}e_{i}\right) \right)$ \\
&$=$&$\sum\limits_{i=1}^{n-3}\frac{2}{i+1}\alpha_{i}e_{i+3}
 +\sum\limits_{i=1}^{n-3}\frac{2}{i+1}\alpha_{i}e_{i+3} = \sum\limits_{i=1}^{n-3}\frac{4}{i+1}\alpha_{i}e_{i+3};$\\
$D(e_{5})$
&$=$&$ 6D(e_{1}\diamond e_{2})
 \ =\  6\left(\left(\sum\limits_{i=1}^{n}\alpha_{i}e_{i}\right)\diamond e_{2}
 +e_{1}\diamond\left(\sum\limits_{i=1}^{n}\beta_{i}e_{i}\right)\right)$ \\
&$=$&$ \sum\limits_{i=1}^{n-4}\frac{2}{i+1}\alpha_{i}e_{i+4}
 +\sum\limits_{i=1}^{n-3}\frac{3}{i+1}\beta_{i}e_{i+3}$ \\
&$=$&$ \sum\limits_{i=1}^{n-4}\frac{2}{i+1}\alpha_{i}e_{i+4}
 +\sum\limits_{i=0}^{n-4}\frac{3}{i+2}\beta_{i+1}e_{i+4}
 \ =\  \frac{3}{2}\beta_{1}e_{4}
 +\sum\limits_{i=1}^{n-4}\left(\frac{2}{i+1}\alpha_{i}
 +\frac{3}{i+2}\beta_{i+1}\right)e_{i+4};$\\
$D(e_{6})$
&$=$&$ 8D(e_{1}\diamond e_{3})
 \ =\  8\left(\left(\sum\limits_{i=1}^{n}\alpha_{i}e_{i}\right)\diamond e_{3}
 +e_{1}\diamond\left(\sum\limits_{i=1}^{n}\gamma_{i}e_{i}\right) \right) $\\
&$=$&$ \sum\limits_{i=1}^{n-5}\frac{2}{i+1}\alpha_{i}e_{i+5}
 +\sum\limits_{i=1}^{n-3}\frac{4}{i+1}\gamma_{i}e_{i+3}$ \\
&$=$&$ \sum\limits_{i=1}^{n-5}\frac{2}{i+1}\alpha_{i}e_{i+5}
 +\sum\limits_{i=-1}^{n-5}\frac{4}{i+3}\gamma_{i+2}e_{i+5}$ \\
&$=$&$ 2\gamma_{1}e_{4}+\frac{4}{3}\gamma_{2}e_{5}
 +\sum\limits_{i=1}^{n-5}\left(\frac{2}{i+1}\alpha_{i}
 +\frac{4}{i+3}\gamma_{i+2}\right)e_{i+5};$\\
$D(e_{6})$
&$=$&$ 9D(e_{2}\diamond e_{2})
\ = \ 9\left(\left(\sum\limits_{i=1}^{n}\beta_{i}e_{i}\right)\diamond e_{2}
 +e_{2}\diamond\left(\sum\limits_{i=1}^{n}\beta_{i}e_{i}\right)\right)$ \\
&$=$&$ \sum\limits_{i=1}^{n-4}\frac{3}{i+1}\beta_{i}e_{i+4}
 +\sum\limits_{i=1}^{n-4}\frac{3}{i+1}\beta_{i}e_{i+4}$ \\
&$=$&$ \sum\limits_{i=1}^{n-4}\frac{6}{i+1}\beta_{i}e_{i+4} =3\beta_{1}e_{5}+\sum\limits_{i=1}^{n-5}\frac{6}{i+2}\beta_{i+1}e_{i+5}.$
\end{longtable}
 
Comparing two different expressions for $D(e_6),$ we have that \begin{center}
$\gamma_{1}=0;$ \ $\gamma_{2}=\frac{9}{4}\beta_{1};$ \ $\gamma_{k}=\frac{3(k+1)}{2k}\beta_{k-1}-\frac{k+1}{2(k-1)}\alpha_{k-2}, \ 3\le k\le n-3;$ \ i.e.
\end{center} \begin{center}
$D(e_{3}) \ = \ \frac{9}{4}\beta_{1}e_{2}+\sum\limits_{i=3}^{n-3}\left(\frac{3(i+1)}{2i}\beta_{i-1}-\frac{i+1}{2(i-1)}\alpha_{i-2}\right)e_{i}+\gamma_{n-2}e_{n-2}+\gamma_{n-1}e_{n-1}+\gamma_{n}e_{n}.$
\end{center} In a similar way, we find two different expressions for $D(e_7):$  
\begin{longtable}{lcl}
$D(e_{7})$ & $=$ & $10D(e_{1}\diamond e_{4}) \ = \ 10\left(\left(\sum\limits_{i=1}^{n}\alpha_{i}e_{i}\right)\diamond e_{4}+e_{1}\diamond\left(\sum\limits_{i=1}^{n-3}\frac{4}{i+1}\alpha_{i}e_{i+3}\right) \right) \ =$ \\

& $=$ & $\sum\limits_{i=1}^{n-6}\frac{2}{i+1}\alpha_{i}e_{i+6}+\sum\limits_{i=1}^{n-6}\frac{20}{(i+1)(i+4)}\alpha_{i}e_{i+6} \ = \ \sum\limits_{i=1}^{n-6}\frac{2(i+14)}{(i+1)(i+4)}\alpha_{i}e_{i+6};$\\ 
$D(e_{7})$ & $=$ & $12D(e_{2}\diamond e_{3}) \ = \ 12 \big(D(e_{2})\diamond e_{3}+e_{2}\diamond D(e_{3}) \big)\ =$ \\

& $=$ & $12 \left(\left(\sum\limits_{i=1}^{n}\beta_{i}e_{i}\right)\diamond e_{3}+\right.$ \\

& & $+\left.e_{2}\diamond\left(\frac{9}{4}\beta_{1}e_{2}+\sum\limits_{i=3}^{n-3}\left(\frac{3(i+1)}{2i}\beta_{i-1}-\frac{i+1}{2(i-1)}\alpha_{i-2}\right)e_{i}+\sum\limits_{j=n-2}^n\gamma_{j}e_{j}\right) \right) \ =$ \\

& $=$ & $\sum\limits_{i=1}^{n-5}\frac{3}{i+1}\beta_{i}e_{i+5}+{3}\beta_{1}e_{6}+\sum\limits_{i=3}^{n-4}\frac{2}{i+1}\left(\frac{3(i+1)}{2i}\beta_{i-1}-\frac{6(i+1)}{i-1}\alpha_{i-2}\right)e_{i+4} \ =$ \\

& $=$ & $\sum\limits_{i=1}^{n-5}\frac{3}{i+1}\beta_{i}e_{i+5}+3\beta_{1}e_{6}+\sum\limits_{i=3}^{n-4}\left(\frac{6}{i}\beta_{i-1}-\frac{2}{i-1}\alpha_{i-2}\right)e_{i+4}=$ \\

& $=$ & $\sum\limits_{i=1}^{n-5}\frac{3}{i+1}\beta_{i}e_{i+5}+3\beta_{1}e_{6}+\sum\limits_{i=2}^{n-5}\left(\frac{6}{i+1}\beta_{i}-\frac{2}{i}\alpha_{i-1}\right)e_{i+5}=$ \\

& $=$ & $\frac{3}{2}\beta_{1}e_{6}+\sum\limits_{i=2}^{n-5}\frac{3}{i+1}\beta_{i}e_{i+5}+3\beta_{1}e_{6}+\sum\limits_{i=2}^{n-5}\left(\frac{6}{i+1}\beta_{i}-\frac{2}{i}\alpha_{i-1}\right)e_{i+5}=$ \\

& $=$ & $\frac{9}{2}\beta_{1}e_{6}+\sum\limits_{i=2}^{n-5}\left(\frac{9}{i+1}\beta_{i}-\frac{2}{i}\alpha_{i-1}\right)e_{i+5}\ =$ \\ 
& $=$ & $\frac{9}{2}\beta_{1}e_{6}+\sum\limits_{i=1}^{n-6}\left(\frac{9}{i+2}\beta_{i+1}-\frac{2}{i+1}\alpha_{i}\right)e_{i+6}.$
\end{longtable}
\noindent  Comparing two different expressions for $D(e_7),$ we have that $\beta_{1}=0$ and 
\begin{center} $\beta_{k}=\frac{4(k+1)(k+8)}{9k(k+3)}\alpha_{k-1}, \ \  2\le k\le n-5;$ \ i.e.,\end{center} 
\begin{longtable}{lcl}
$D(e_{2}) \ = \ \sum\limits_{i=2}^{n-5}\frac{4(i+1)(i+8)}{9i(i+3)}\alpha_{i-1}e_{i}+\sum\limits_{j=n-4}^n\beta_{j}e_{j};$ \\

$D(e_{3}) \ = \ \sum\limits_{i=3}^{n-4}\frac{(i+1)(i+22)}{6(i-1)(i+2)}\alpha_{i-2}e_{i}+\left(\frac{3(n-2)}{2(n-3)}\beta_{n-4}-\frac{n-2}{2(n-4)}\alpha_{n-5}\right)e_{n-3}+\sum\limits_{j=n-2}^n\gamma_{j}e_{j};$ \\

$D(e_{5}) \ = \ \sum\limits_{i=1}^{n-6}\frac{10(i+6)}{3(i+1)(i+4)}\alpha_{i}e_{i+4}+\left(\frac{2}{n-4}\alpha_{n-5}-\frac{3}{n-3}\beta_{n-4}\right)e_{n-1}+\left(\frac{2}{n-3}\alpha_{n-4}-\frac{3}{n-2}\beta_{n-3}\right)e_{n};$ \\

$D(e_{6}) \ = \ \sum\limits_{i=1}^{n-6}\frac{8(i+9)}{3(i+1)(i+4)}\alpha_{i}e_{i+5}+\frac{6}{n-3}\beta_{n-4}e_{n}.$
\end{longtable}
\noindent  Next, we obtain two expressions for $D(e_8):$ 
\begin{longtable}{lcl}
$D(e_{8})$ & $=$ & $12D(e_{1}\diamond e_{5}) \ = \ 12\left(\sum\limits_{i=1}^{n}\alpha_{i}e_{i}\right)\diamond e_{5}+12e_{1}\diamond\left(\sum\limits_{i=1}^{n-6}\frac{10(i+6)}{3(i+1)(i+4)}\alpha_{i}e_{i+4}+\right.$ \\
& & $\left.+\left(\frac{2\alpha_{n-5}}{n-4}-\frac{3\beta_{n-4}}{n-3}\right)e_{n-1}+\left(\frac{2\alpha_{n-4}}{n-3}-\frac{3\beta_{n-3}}{n-2}\right)e_{n}\right) \ =$ \\

& $=$ & $\sum\limits_{i=1}^{n-7}\frac{2}{i+1}\alpha_{i}e_{i+7}+
\sum\limits_{i=1}^{n-7}\frac{20(i+6)}{(i+1)(i+4)(i+5)}\alpha_{i}e_{i+7} \ = \ \sum\limits_{i=1}^{n-7}\frac{2(i^{2}+19i+80)}{(i+1)(i+4)(i+5)}\alpha_{i}e_{i+7};$\\

$D(e_{8})$ & $=$ & $16D(e_{3}\diamond e_{3}) \ = \ 
16\left(\sum\limits_{i=3}^{n-4}\frac{(i+1)(i+22)}{6(i-1)(i+2)}\alpha_{i-2}e_{i}+\right.$ \\

& & $\left.+\left(\frac{3(n-2)}{2(n-3)}\beta_{n-4}-\frac{n-2}{2(n-4)}\alpha_{n-5}\right)e_{n-3}+\sum\limits_{j=n-2}^n\gamma_{j}e_{j}\right)\diamond e_{3}+$ \\

& & $+16e_{3}\diamond\left(\sum\limits_{i=3}^{n-4}\frac{(i+1)(i+22)}{6(i-1)(i+2)}\alpha_{i-2}e_{i}+\right.$ \\
& & $\left.+\left(\frac{3(n-2)}{2(n-3)}\beta_{n-4}-\frac{n-2}{2(n-4)}\alpha_{n-5}\right)e_{n-3}+ \sum\limits_{j=n-2}^n\gamma_{j}e_{j} \right) \ =$ \\

& $=$ & $\sum\limits_{i=3}^{n-5}\frac{2(i+22)}{3(i-1)(i+2)}\alpha_{i-2}e_{i+5}+
\sum\limits_{i=3}^{n-5}\frac{2(i+22)}{3(i-1)(i+2)}\alpha_{i-2}e_{i+5}$ \\ 
& $=$ & $\sum\limits_{i=1}^{n-7}\frac{4(i+24)}{3(i+1)(i+4)}\alpha_{i}e_{i+7}.$
\end{longtable}
\noindent  Comparing the present expressions for $D(e_8),$ we have $\alpha_{k}=0,$ for $2\le k\le n-7,$ that gives  
\begin{longtable}{lcl}
$D(e_{1}) \ = \ \alpha_{1}e_{1}+\sum\limits_{j=n-6}^n\alpha_{j}e_{j};$ \\

$D(e_{2}) \ = \ \frac{4}{3}\alpha_{1}e_{2}+\frac{4(n-4)(n+3)}{9(n-2)(n-5)}\alpha_{n-6}e_{n-5}+
\sum\limits_{j=n-4}^n\beta_{j}e_{j};$ \\

$D(e_{3}) \ = \ \frac{5}{3}\alpha_{1}e_{3}+\frac{(n-3)(n+18)}{6(n-2)(n-5)}\alpha_{n-6}e_{n-4}+\left(\frac{3(n-2)}{2(n-3)}\beta_{n-4}-\frac{n-2}{2(n-4)}\alpha_{n-5}\right)e_{n-3}+
\sum\limits_{j=n-2}^n\gamma_{j}e_{j};$ \\

$D(e_{4}) \ = \ 2\alpha_{1}e_{4}+\frac{4}{n-5}\alpha_{n-6}e_{n-3}+\frac{4}{n-4}\alpha_{n-5}e_{n-2}+\frac{4}{n-3}\alpha_{n-4}e_{n-1}+\frac{4}{n-2}\alpha_{n-3}e_{n};$ \\

$D(e_{5}) \ = \ \frac{7}{3}\alpha_{1}e_{5}+\frac{10n\alpha_{n-6}}{3(n-2)(n-5)}e_{n-2}+\left(\frac{2\alpha_{n-5}}{n-4}-\frac{3\beta_{n-4}}{n-3}\right)e_{n-1}+\left(\frac{2\alpha_{n-4}}{n-3}-\frac{3\beta_{n-3}}{n-2}\right)e_{n};$ \\

$D(e_{6}) \ = \ \frac{8}{3}\alpha_{1}e_{6}+\frac{8(n+3)}{3(n-2)(n-5)}\alpha_{n-6}e_{n-1}+\frac{6}{n-3}\beta_{n-4}e_{n};$ \\

$D(e_{7}) \ = \ 3\alpha_{1}e_{7}+\frac{2(n+8)}{(n-2)(n-5)}\alpha_{n-6}e_{n};$ \\

$D(e_{8}) \ = \ \frac{10}{3}\alpha_{1}e_{8}.$ \\
\end{longtable}

Furthermore, using the induction method, we can prove that $D(e_{k})=\frac{k+2}{3}\alpha_{1}e_{k},$ for $8 \le k\le n.$ To do this, let us consider
 
\begin{longtable}{lcl}
$D(e_{k+1})$ & $=$ & $2(k-1)D(e_{1}\diamond e_{k-2}) \ = \ 2(k-1) \big(D(e_{1})\diamond e_{k-2}+e_ {1}\diamond D(e_{k-2})\big) \ =$ \\ 

& $=$ & $2(k-1) \left(\left(\alpha_{1}e_{1}+\sum\limits_{j=n-6}^n\alpha_{j}e_{j}\right)\diamond e_{k-2}+e_{1}\diamond \left(\frac{k}{3}\alpha_{1}e_{k-2}\right) \right) \ = \ \frac{k+3}{3}\alpha_{1}e_{k+1}.$
\end{longtable}
\noindent Combining all the obtained results, we see that $D$ is a linear combination of mappings
$\left\{\varphi, \phi_i, \psi_j, \pi_k\right\}_{ 0\le i \le 6; \ 0\le j \le 4}^{ 0\le k \le 2}.$
\end{proof}

\begin{corollary}
$\mathfrak{Der}\big({\rm IDD}(K^1_\infty,-1,-1)\big) = \big\langle \varphi\big\rangle,$ where $\varphi(e_{i}) = \left(i+2\right) e_{i}.$
\end{corollary}

\subsection{Derivations of \texorpdfstring{$\mathrm{IDD}(K^\times_\mathfrak{n},0,-2)$}{IDD(K\^x\_n,0,-2)}}\label{der0-2}

\begin{proposition}
$\mathfrak{Der}\big({\rm IDD}(K^0_2,0,-2)\big) = \big\langle \varphi_i, \phi_j \big\rangle_{ 0\le i \le 2}^{1\le j \le 2},$ where  
\begin{longtable}{|c|c|l|}
\hline 
$\varphi_i$ & $\varphi_i(e_{0}) \ = \ e_{i}, \ \varphi_i(e_{2})\ = \ 2 {\delta}_{0,i}e_{2}$ & $0\leq i\leq 2$ \\
\hline
$\phi_j$ & $\phi_j(e_{1}) \ = \ e_{j}$ & $1\leq j\leq 2$ \\
\hline
\end{longtable}
 
\end{proposition}

\begin{proposition}
$\mathfrak{Der}\big({\rm IDD}(K^0_3,0,-2)\big) = \big\langle \varphi_i, \phi_j \big\rangle_{ 0\le i \le 3}^{1\le j \le 3},$ where  
\begin{longtable}{|c|c|}
\hline
$\varphi_{3}$ & $\varphi_{3}(e_{0})=e_{3}$ \\
\hline
$\varphi_{2}$ & $\varphi_{2}(e_{0})=e_{2}$ \\
\hline
$\varphi_{1}$ & $\varphi_{1}(e_{0})= 3 e_{1},$ \ $\varphi_{1}(e_{2})= 4 e_{3}$ \\
\hline
$\varphi_{0}$ & $\varphi_{0}(e_{0})=e_{0},$ \ $\varphi_{0}(e_{2})=2e_{2},$ \ $\varphi_{0}(e_{3})=e_{3}$ \\
\hline
$\phi_{3}$ & $\phi_{3}(e_{1})=e_{3}$ \\
\hline
$\phi_{2}$ & $\phi_{2}(e_{1})=e_{2}$ \\
\hline
$\phi_{1}$ & $\phi_{1}(e_{1})=e_{1},$ \ $\phi_{1}(e_{3})=e_{3}$ \\
\hline
\end{longtable}
 
\end{proposition}

\begin{theorem}
If $4\le n < \infty,$ then $\mathfrak{Der}\big({\rm IDD}(K^0_n,0,-2)\big) = \big\langle \varphi, \phi_i, \psi_j, \pi_k\big\rangle_{ 0\le i \le 3}^{0\le j \le 1},$ where  
\begin{longtable}{|c|c|}
\hline 
$\varphi$ & $\varphi(e_{i}) \ = \ \left(i+2\right) e_{i}, \ 0\leq i\leq n$ \\
\hline
$\phi_3$ & $\phi_3(e_0) = (n-1)(n^2-4) e_{n-3}, \  
\phi_3(e_1) = n^{2}(n-1) e_{n-2},$ \\
& $\phi_3(e_2) = (n+2)(n^2-3n+4) e_{n-1}, \
\phi_3(e_3) = (n+1)(n^2-2n+4) e_{n}$ \\
\hline
$\phi_2$ & $\phi_2(e_0) = n(n-1) e_{n-2},\  
\phi_2(e_2) = (n^2-n+2)e_{n}$ \\
\hline
$\phi_1$ & $\phi_1(e_0) = e_{n-1}$ \\
\hline
$\phi_0$ & $\phi_0(e_0) = e_{n}$ \\
\hline
$\psi_1$ & $\psi_1(e_1) = e_{n-1}$ \\
\hline
$\psi_0$ & $\psi_0(e_1) = e_{n}$ \\
\hline
\end{longtable}
 
\end{theorem}

\begin{proof}
Let $D \in \mathfrak{Der}\big({\rm IDD}(K^0_n,0,-2)\big),$ then $D(e_{0})=\sum\limits_{i=0}^{n}\alpha_{i}e_{i}$ and $D(e_{1})=\sum\limits_{i=0}^{n}\beta_{i}e_{i}.$ Firstly,  
\begin{longtable}{lcl}
$D(e_{2})$ & $=$ & $2D(e_{0}\diamond e_{0}) \ = \ 2\left(\left(\sum\limits_{i=0}^{n}\alpha_{i}e_{i}\right)\diamond e_{0}+e_{0}\diamond\left(\sum\limits_{i=0}^{n}\alpha_{i}e_{i}\right) \right) \ =$ \\

& $=$ & $\sum\limits_{i=0}^{n-2}\alpha_{i}e_{i+2}+\sum\limits_{i=0}^{n-2}\frac{2}{(i+1)(i+2)}\alpha_{i}e_{i+2}\ = \ \sum\limits_{i=0}^{n-2}\frac{i^{2}+3i+4}{(i+1)(i+2)}\alpha_{i}e_{i+2};$ \\

$D(e_{3})$ & $=$ & $6D(e_{0}\diamond e_{1}) \ = \ 6\left(\left(\sum\limits_{i=0}^{n}\alpha_{i}e_{i}\right)\diamond e_{1}+e_{0}\diamond\left(\sum\limits_{i=0}^{n}\beta_{i}e_{i}\right)\right) \ = $ \\

& $=$ & $\sum\limits_{i=0}^{n-3}\alpha_{i}e_{i+3}+\sum\limits_{i=0}^{n-2}\frac{6}{(i+1)(i+2)}\beta_{i}e_{i+2}\ = \ 3\beta_{0}e_{2}+\sum\limits_{i=0}^{n-3}\left(\alpha_{i}+\frac{6}{(i+2)(i+3)}\beta_{i+1}\right)e_{i+3};$ \\

$D(e_{3})$ & $=$ & $2D(e_{1}\diamond e_{0}) \ = \ 2 \left(\left(\sum\limits_{i=0}^{n}\beta_{i}e_{i}\right)\diamond e_{0}+e_{1}\diamond\left(\sum\limits_{i=0}^{n}\alpha_{i}e_{i}\right) \right) \ =$ \\

& $=$ & $\sum\limits_{i=0}^{n-2} \beta_{i}e_{i+2}+\sum\limits_{i=0}^{n-3}\frac{2}{(i+1)(i+2)}\alpha_{i}e_{i+3} \ = \ \beta_{0}e_{2}+\sum\limits_{i=0}^{n-3}\left(\beta_{i+1}+\frac{2}{(i+1)(i+2)}\alpha_{i}\right)e_{i+3};$ \\
\end{longtable}
\noindent that immediately gives $\beta_{0}=0$ and $\beta_{k}=\frac{(k+2)^{2}}{k(k+4)}\alpha_{k-1}$ for $2\le k\le n-2,$ i.e.,  
\begin{longtable}{lcl}
$D(e_{1})$ & $=$ & $\beta_{1}e_{1}+\sum\limits_{i=2}^{n-2}\frac{(i+2)^{2}}{i(i+4)}\alpha_{i-1}e_{i}+\beta_{n-1}e_{n-1}+\beta_{n}e_{n}$ and \\
$D(e_{3})$ & $=$ & $\left(\alpha_{0}+\beta_{1}\right)e_{3}+\sum\limits_{i=1}^{n-3}\frac{(i+4)(i^2+4i+7)}{(i+1)(i+2)(i+5)}\alpha_{i}e_{i+3}.$
\end{longtable}
\noindent  Secondly,  
\begin{longtable}{lcl}
$D(e_{4})$ & $=$ & $12D(e_{0}\diamond e_{2}) \ = \ 12 \left(\left(\sum\limits_{i=0}^{n}\alpha_{i}e_{i}\right)\diamond e_{2}+e_{0}\diamond\left(\sum\limits_{i=0}^{n-2}\frac{i^{2}+3i+4}{(i+1)(i+2)}\alpha_{i}e_{i+2}\right)\right) \ = $ \\

& $=$ & $\sum\limits_{i=0}^{n-4}\alpha_{i}e_{i+4}+\sum\limits_{i=0}^{n-4}\frac{12(i^{2}+3i+4)}{(i+1)(i+2)(i+3)(i+4)}\alpha_{i}e_{i+4} \ =  \sum\limits_{i=0}^{n-4}\frac{i^4+10i^3+47i^{2}+86i+72}{(i+1)(i+2)(i+3)(i+4)}\alpha_{i}e_{i+4};$ \\

$D(e_{4})$ & $=$ & $2D(e_{2}\diamond e_{0}) \ = \ 2\left( \left(\sum\limits_{i=0}^{n-2}\frac{i^{2}+3i+4}{(i+1)(i+2)}\alpha_{i}e_{i+2}\right)\diamond e_{0}+e_{2}\diamond\left(\sum\limits_{i=0}^{n}\alpha_{i}e_{i}\right)\right) \ =$ \\

& $=$ & $\sum\limits_{i=0}^{n-4}\frac{i^{2}+3i+4}{(i+1)(i+2)}\alpha_{i}e_{i+4}+\sum\limits_{i=0}^{n-4}\frac{2}{(i+1)(i+2)}\alpha_{i}e_{i+4} \ = \ \sum\limits_{i=0}^{n-4}\frac{i^{2}+3i+6}{(i+1)(i+2)}\alpha_{i}e_{i+4};$\\
\end{longtable}
\noindent that immediately gives $\alpha_{k}=0$ for $1\le k\le n-4,$ i.e.,  
\begin{longtable}{lcl}
$D(e_{0})$ & $=$ & $\alpha_{0}e_{0}+\alpha_{n-3}e_{n-3}+\alpha_{n-2}e_{n-2}+\alpha_{n-1}e_{n-1}+\alpha_{n}e_{n};$ \\

$D(e_{1})$ & $=$ & $\beta_{1}e_{1}+\frac{n^2}{n^2-4}\alpha_{n-3}e_{n-2}+\beta_{n-1}e_{n-1}+\beta_{n}e_{n};$ \\

$D(e_{2})$ & $=$ & $2\alpha_{0}e_{2}+\frac{n^2-3n+4}{(n-1)(n-2)}\alpha_{n-3}e_{n-1}+\frac{n^2-n+2}{n(n-1)}\alpha_{n-2}e_{n};$ \\

$D(e_{3})$ & $=$ & $\left(\alpha_{0}+\beta_{1}\right)e_{3}+\frac{(n+1)(n^2-2n+4)}{(n-1)(n^2-4)}\alpha_{n-3}e_{n};$ \\
$D(e_{4})$ & $=$ & $3\alpha_{0}e_{4}.$ \\
\end{longtable}
\noindent On the other hand,  
\begin{longtable}{lcl}
$D(e_{4})$ & $=$ & $6D(e_{1}\diamond e_{1}) \ = \ 
6\left(\left(\beta_{1}e_{1}+\frac{n^2\alpha_{n-3}}{n^2-4}e_{n-2}+\sum\limits_{i=n-1}^n\beta_{i}e_{i}\right)\diamond e_{1}+\right. $\\

& & $\left.+e_{1}\diamond\left(\beta_{1}e_{1}+\frac{n^2\alpha_{n-3}}{n^2-4} e_{n-2}+\sum\limits_{i=n-1}^n\beta_{i}e_{i}\right)\right) \ = \ 2\beta_{1}e_{4};$ \\
\end{longtable}
\noindent  i.e., $\beta_{1}=\frac{3}{2}\alpha_{0}.$ Then we obtain  
\begin{longtable}{lcl}
$D(e_{1})$ & $=$ & $\frac{3}{2}\alpha_{0}e_{1}+\frac{n^2}{n^2-4}\alpha_{n-3}e_{n-2}+\beta_{n-1}e_{n-1}+\beta_{n}e_{n};$ \\

$D(e_{3})$ & $=$ & $\frac{5}{2}\alpha_{0}e_{3}+\frac{(n+1)(n^2-2n+4)}{(n-1)(n^2-4)}\alpha_{n-3}e_{n}.$ \\
\end{longtable}

Furthermore, using the induction method, we can prove that $D(e_{k})=\frac{k+2}{2}\alpha_{0}e_{k},$ for $4 \le k\le n.$ To do this, let us consider  
\begin{longtable}{lcl}
$D(e_{k+2})$ & $=$ & $2D(e_{k}\diamond e_{0}) \ = \ 
2\big(D(e_{k})\diamond e_{0}+e_{k}\diamond D(e_{0})\big) \ = $\\
& $=$ & $2\left(\frac{k+2}{2}\alpha_{0}e_{k}\right)\diamond e_{0}+e_{k}\diamond\left(\alpha_{0}e_{0}+\sum\limits_{i=n-3}^n\alpha_{i}e_{i}\right) \ = \ \frac{k+4}{4}\alpha_{0}e_{k+2}.$ \\
\end{longtable}
\noindent  Combining all the obtained results, we see that $D$ is a linear combination of mappings
$\left\{ \varphi, \phi_i, \psi_j\right\}_{ 0\le i \le 3; \ 0\le j \le 1}.$
\end{proof}

\begin{corollary}
$\mathfrak{Der}\big({\rm IDD}(K^0_\infty,0,-2)\big) = \big\langle \varphi\big\rangle,$ where $\varphi(e_{i}) =  \left(i+2\right) e_{i}.$
\end{corollary}

\begin{proposition}
$\mathfrak{Der}\big({\rm IDD}(K^1_4,0,-2)\big) = \big\langle \varphi_i, \phi_j, \psi_j\big\rangle_{1\le i\le 4}^{2\le j\le 4},$ where  
\begin{longtable}{|c|c|l|}
\hline
$\varphi_i$ & $\varphi_{i}(e_k)=\delta_{1,k}e_{i}+2\delta_{1,i}\delta_{4,k}e_{4}$ & $1\le i\le 4$ \\
\hline
$\phi_j$ & $\phi_{j}(e_{2})=e_{j}$ & $2\le j\le 4$ \\
\hline
$\psi_j$ & $\psi_{j}(e_{3})=e_{j}$ & $2\le j\le 4$ \\
\hline
\end{longtable}
 
\end{proposition}

\begin{proposition}
$\mathfrak{Der}\big({\rm IDD}(K^1_5,0,-2)\big) =  \big\langle \varphi_i, \phi_j, \psi_k\big\rangle_{1\le i\le 5; \ 2\le j\le 5}^{3\le k\le 5},$ where  
\begin{longtable}{|c|c|}
\hline
$\varphi_{5}$ & $\varphi_{5}(e_{1})=e_{5}$ \\
\hline
$\varphi_{4}$ & $\varphi_{4}(e_{1})=e_{4}$ \\
\hline
$\varphi_{3}$ & $\varphi_{3}(e_{1})=e_{3}$ \\
\hline
$\varphi_{2}$ & $\varphi_{2}(e_{1})= 2 e_{2},$ \ $\varphi_{2}(e_{4})= 3 e_{5}$ \\
\hline
$\varphi_{1}$ & $\varphi_{1}(e_{1})=e_{1},$ \ $\varphi_{1}(e_{4})=2e_{4},$ \ $\varphi_{1}(e_{5})=e_{5}$ \\
\hline
$\phi_{5}$ & $\phi_{5}(e_{2})=e_{5}$ \\
\hline
$\phi_{4}$ & $\phi_{4}(e_{2})=e_{4}$ \\
\hline
$\phi_{3}$ & $\phi_{3}(e_{2})=e_{3}$ \\
\hline
$\phi_{2}$ & $\phi_{2}(e_{2})=e_{2},$ \ $\varphi_{2}(e_{5})=e_{5}$ \\
\hline
$\psi_{5}$ & $\psi_{5}(e_{3})=e_{5}$ \\
\hline
$\psi_{4}$ & $\psi_{4}(e_{3})=e_{4}$ \\
\hline
$\psi_{3}$ & $\psi_{3}(e_{3})=e_{3}$ \\
\hline
\end{longtable}
 
\end{proposition}

\begin{proposition}
$\mathfrak{Der}\big({\rm IDD}(K^1_6,0,-2)\big) =  \big\langle \varphi_i, \phi_j, \psi_k\big\rangle_{1\le i\le 6; \ 2\le j\le 6}^{4\le k\le 6},$ where  
\begin{longtable}{|c|c|}
\hline
$\varphi_{6}$ & $\varphi_{6}(e_{1})=e_{6}$ \\
\hline
$\varphi_{5}$ & $\varphi_{5}(e_{1})=e_{5}$ \\
\hline
$\varphi_{4}$ & $\varphi_{4}(e_{1})=e_{4}$ \\
\hline
$\varphi_{3}$ & $\varphi_{3}(e_{1})= 10 e_{3},$ \ $\varphi_{3}(e_{4})= 13 e_{6}$ \\
\hline
$\varphi_{2}$ & $\varphi_{2}(e_{1})= 4 e_{2},$ \ $\varphi_{2}(e_{2})= 5 e_{3},$ \ $\varphi_{2}(e_{4})= 6 e_{5},$ \ $\varphi_{2}(e_{5})= 7 e_{6}$ \\
\hline
$\varphi_{1}$ & $\varphi_{1}(e_{1})=e_{1},$ \ $\varphi_{1}(e_{3})=-e_{3},$ \ $\varphi_{1}(e_{4})=2e_{4},$ \ $\varphi_{1}(e_{5})=e_{5}$ \\
\hline
$\phi_{6}$ & $\phi_{6}(e_{2})=e_{6}$ \\
\hline
$\phi_{5}$ & $\phi_{5}(e_{2})=e_{5}$ \\
\hline
$\phi_{4}$ & $\phi_{4}(e_{2})=e_{4}$ \\
\hline
$\phi_{3}$ & $\phi_{3}(e_{2})=e_{3}$ \\
\hline
$\phi_{2}$ & $\phi_{2}(e_{2})=e_{2},$ \ $\phi_{2}(e_{3})=2e_{3},$ \ $\phi_{2}(e_{5})=e_{5},$ \ $\phi_{2}(e_{6})=2e_{6}$ \\
\hline
$\psi_{6}$ & $\psi_{6}(e_{3})=e_{6}$ \\
\hline
$\psi_{5}$ & $\psi_{5}(e_{3})=e_{5}$ \\
\hline
$\psi_{4}$ & $\psi_{4}(e_{3})=e_{4}$ \\
\hline
\end{longtable}
 
\end{proposition}

\begin{proposition}
$\mathfrak{Der}\big({\rm IDD}(K^1_7,0,-2)\big) =  \big\langle \varphi_i, \phi_j, \psi_k\big\rangle_{1\le i\le 7; \ 5\le j\le 7}^{5\le k\le 7},$ where  
\begin{longtable}{|c|c|}
\hline
$\varphi_{7}$ & $\varphi_{7}(e_{1})=e_{7}$ \\
\hline
$\varphi_{6}$ & $\varphi_{6}(e_{1})=e_{6}$ \\
\hline
$\varphi_{5}$ & $\varphi_{5}(e_{1})=e_{5}$ \\
\hline
$\varphi_{4}$ & $\varphi_{4}(e_{1})= 5 e_{4},$ \ $\varphi_{4}(e_{4})= 6 e_{7}$ \\
\hline
$\varphi_{3}$ & $\varphi_{3}(e_{1})= 30 e_{3},$ \ $\varphi_{3}(e_{2})= 35 e_{4},$ \ $\varphi_{3}(e_{4})= 39 e_{6},$ \ $\varphi_{3}(e_{5})= 44 e_{7}$ \\
\hline
$\varphi_{2}$ & $\varphi_{2}(e_{1})= 20 e_{2},$ \ $\varphi_{2}(e_{2})= 25 e_{3},$ \ $\varphi_{2}(e_{3})= 30 e_{4},$ \\
& $\varphi_{2}(e_{4})= 30 e_{5},$ \ $\varphi_{2}(e_{5})= 28 e_{6},$ \ $\varphi_{2}(e_{6})= 40 e_{7}$ \\
\hline
$\varphi_{1}$ & $\varphi_{1}(e_{1})= 3 e_{1},$ \ $\varphi_{1}(e_{2})= 4 e_{2},$ \ $\varphi_{1}(e_{3})= 5 e_{3},$ \\
& $\varphi_{1}(e_{4})= 6 e_{4},$ \ $\varphi_{1}(e_{5})= 7 e_{5},$ \ $\varphi_{1}(e_{6})= 8 e_{6},$ \ $\varphi_{1}(e_{7})= 9 e_{7}$ \\
\hline
 $\phi_{7}$ & $\phi_{7}(e_{2})=e_{7}$ \\
\hline
$\phi_{6}$ & $\phi_{6}(e_{2})=e_{6}$ \\
\hline
$\phi_{5}$ & $\phi_{5}(e_{2})=e_{5}$ \\
\hline
$\psi_{7}$ & $\psi_{7}(e_{3})=e_{7}$ \\
\hline
$\psi_{6}$ & $\psi_{6}(e_{3})=e_{6}$ \\
\hline
$\psi_{5}$ & $\psi_{5}(e_{3})=e_{5}$ \\
\hline
\end{longtable}
 
\end{proposition}

\begin{theorem}
If $8\le n < \infty,$ then $\mathfrak{Der}\big({\rm IDD}(K^1_n,0,-2)\big) = \big\langle \varphi, \phi_i, \psi_j, \pi_k\big\rangle_{ 0\le i \le 4}^{0\le j , k \le 2},$ where  
\begin{longtable}{|c|c|}
\hline 
$\varphi$ & $\varphi(e_{i}) \ = \ \left(i+2\right) e_{i}, \ 1\leq i\leq n$ \\
\hline
$\phi_4$ & $\phi_4(e_1) = (n-3)(n^{2}-4) e_{n-4}, \  
\phi_4(e_2) = n(n-1)(n-2) e_{n-3},$ \\
& $\phi_4(e_4) = (n+2)(n^{2}-5n+12) e_{n-1}, \
\phi_4(e_5) = (n+1)(n^{2}-4n+12) e_{n}$ \\
\hline
$\phi_3$ & $\phi_3(e_1) = (n-1)(n-2) e_{n-3}, \
\phi_3(e_4) = (n^{2}-3n+8) e_{n}$ \\
\hline
$\phi_2$ & $\phi_2(e_1) = e_{n-2}$ \\
\hline
$\phi_1$ & $\phi_1(e_1) = e_{n-1}$ \\
\hline
$\phi_0$ & $\phi_0(e_1) = e_{n}$  \\
\hline
$\psi_2$ & $\psi_2(e_2) = e_{n-2}$ \\
\hline
$\psi_1$ & $\psi_1(e_2) = e_{n-1}$ \\
\hline
$\psi_0$ & $\psi_0(e_2) = e_{n}$ \\
\hline
$\pi_2$ & $\pi_2(e_3) = e_{n-2}$ \\
\hline
$\pi_1$ & $\pi_1(e_3) = e_{n-1}$ \\
\hline
$\pi_0$ & $\pi_0(e_3) = e_{n}$ \\
\hline
\end{longtable}
\end{theorem}

\begin{proof}
Let $D \in \mathfrak{Der}\big({\rm IDD}(K^1_n,0,-2)\big),$ then we can say that 
\begin{center} $D(e_{1})=\sum\limits_{i=1}^{n}\alpha_{i}e_{i},$ \ $D(e_{2})=\sum\limits_{i=1}^{n}\beta_{i}e_{i},$ and $D(e_{3})=\sum\limits_{i=1}^{n}\gamma_{i}e_{i}.$ \end{center} Following the ideas presented above, we find $D(e_4)$ and $D(e_5).$ 
\begin{longtable}{lcl}
$D(e_{4})$ & $=$ & $6D(e_{1}\diamond e_{1}) \ = \
6\left(\left(\sum\limits_{i=1}^{n}\alpha_{i}e_{i}\right)\diamond e_{1}+e_{1}\diamond\left(\sum\limits_{i=1}^{n}\alpha_{i}e_{i}\right) \right)\ =$ \\

& $=$ & $ \sum\limits_{i=1}^{n-3} \alpha_{i}e_{i+3}+\sum\limits_{i=1}^{n-3}\frac{6}{(i+1)(i+2)}\alpha_{i}e_{i+3} \ = \sum\limits_{i=1}^{n-3}\frac{i^{2}+3i+8}{(i+1)(i+2)}\alpha_{i}e_{i+3};$\\

$D(e_{5})$&$=$&$12 D(e_{1}\diamond e_{2})\ = \ 12\left(\left(\sum\limits_{i=1}^{n}\alpha_{i}e_{i}\right)\diamond e_{2}+e_{1}\diamond\left(\sum\limits_{i=1}^{n}\beta_{i}e_{i}\right) \right)\ =$ \\

& $=$ & $\sum\limits_{i=1}^{n-4} \alpha_{i}e_{i+4}+\sum\limits_{i=1}^{n-3}\frac{12}{(i+1)(i+2)}\beta_{i}e_{i+3} \ = \ 2\beta_{1}e_{4}+\sum\limits_{i=1}^{n-4}\left( \alpha_{i}+\frac{12}{(i+2)(i+3)}\beta_{i+1}\right)e_{i+4};$ \\

$D(e_{5})$ & $=$ & $6D(e_{2}\diamond e_{1}) \ = \ 6 \left( \left(\sum\limits_{i=1}^{n}\beta_{i}e_{i}\right)\diamond e_{1}+e_{2}\diamond\left(\sum\limits_{i=1}^{n}\alpha_{i}e_{i}\right) \right) \ =$ \\

& $=$ & $\sum\limits_{i=1}^{n-3}\beta_{i}e_{i+3}+\sum\limits_{i=1}^{n-4}\frac{6}{(i+1)(i+2)}\alpha_{i}e_{i+4}\ =\ \beta_{1}e_{4}+\sum\limits_{i=1}^{n-4}\left(\beta_{i+1}+\frac{6}{(i+1)(i+2)}\alpha_{i}\right)e_{i+4}.$
\end{longtable}
\noindent Comparing two expressions of $D(e_5),$ we have \begin{center} $\beta_{1}=0;$ \
$\beta_{k}=\frac{(k-2)(k+2)(k+3)}{k(k^{2}+3n-10)}\alpha_{k-1}, \ 3\le k\le n-3;$ i.e.,
\end{center}
 \begin{longtable}{lcl}
$D(e_{2})$ & $=$ & $\beta_{2}e_{2}+\sum\limits_{i=3}^{n-3}\frac{(i-2)(i+2)(i+3)}{i(i^{2}+3i-10)}\alpha_{i-1}e_{i}+\beta_{n-2}e_{n-2}+\beta_{n-1}e_{n-1}+\beta_{n}e_{n};$ \\

$D(e_{5})$ & $=$ & $\left(\alpha_{1}+\beta_{2}\right)e_{5}+\sum\limits_{i=2}^{n-4}\frac{(i+5)(i^{2}+4i+12)}{(i+1)(i+2)(i+6)}\alpha_{i}e_{i+4}.$
\end{longtable}
\noindent  Next, we find three different expressions of $D(e_6).$  
\begin{longtable}{lcl}
$D(e_{6})$ & $=$ & $20D(e_{1}\diamond e_{3}) \ = \ 20 \left(\left(\sum\limits_{i=1}^{n}\alpha_{i}e_{i}\right)\diamond e_{3}+e_{1}\diamond\left(\sum\limits_{i=1}^{n}\gamma_{i}e_{i}\right) \right) \ = $ \\

& $=$ & $\sum\limits_{i=1}^{n-5} \alpha_{i}e_{i+5}+\sum\limits_{i=1}^{n-3}\frac{20}{(i+1)(i+2)}\gamma_{i}e_{i+3}\ = \ \sum\limits_{i=1}^{n-5} \alpha_{i}e_{i+5}+\sum\limits_{i=-1}^{n-5}\frac{20}{(i+3)(i+4)}\gamma_{i+2}e_{i+5}=$ \\

& $=$ & $\frac{10}{3}\gamma_{1}e_{4}+\frac{5}{3}\gamma_{2}e_{5}+\sum\limits_{i=1}^{n-5}\left( \alpha_{i}+\frac{20}{(i+3)(i+4)}\gamma_{i+2}\right)e_{i+5};$ \\

$D(e_{6})$ & $=$ & $6D(e_{3}\diamond e_{1})\ = \ 6 \left(\left(\sum\limits_{i=1}^{n}\gamma_{i}e_{i}\right)\diamond e_{1}+e_{3}\diamond\left(\sum\limits_{i=1}^{n}\alpha_{i}e_{i}\right) \right) \ =$ \\

& $=$ & $\sum\limits_{i=1}^{n-3} \gamma_{i}e_{i+3}+\sum\limits_{i=1}^{n-5}\frac{6}{(i+1)(i+2)}\alpha_{i}e_{i+5}\ = \ \sum\limits_{i=-1}^{n-5} \gamma_{i+2}e_{i+5}+\sum\limits_{i=1}^{n-5}\frac{6}{(i+1)(i+2)}\alpha_{i}e_{i+5}\ =$ \\

& $=$ & $\gamma_{1}e_{4}+\gamma_{2}e_{5}+\sum\limits_{i=1}^{n-5}\left( \gamma_{i+2}+\frac{6}{(i+1)(i+2)}\alpha_{i}\right)e_{i+5};$ \\

$D(e_{6})$ & $=$ & $12D(e_{2}\diamond e_{2}) \ = \ 12 \big(D(e_{2})\diamond e_{2}+e_{2}\diamond D(e_{2}) \big) \ =$ \\

& $=$ & $12\left(\beta_{2}e_{2}+\sum\limits_{i=3}^{n-3}\frac{(i-2)(i+2)(i+3)}{i(i^{2}+3i-10)}\alpha_{i-1}e_{i}+\sum\limits_{j=n-2}^n\beta_{j}e_{j} \right)\diamond e_{2}+$ \\

& & $+12e_{2}\diamond\left(\beta_{2}e_{2}+\sum\limits_{i=3}^{n-3}\frac{(i-2)(i+2)(i+3)}{i(i^{2}+3i-10)}\alpha_{i-1}e_{i}+\sum\limits_{j=n-2}^n\beta_{j}e_{j} \right) \ = $ \\

& $=$ & $\beta_{2}e_{6}+\sum\limits_{i=3}^{n-4}\frac{(i-2)(i+2)(i+3)}{i(i^{2}+3i-10)}\alpha_{i-1}e_{i+4}+ \beta_{2}e_{6}+\sum\limits_{i=3}^{n-4}\frac{12(i-2)(i+3)}{i(i+1)(i^{2}+3i-10)}\alpha_{i-1}e_{i+4}\ =$ \\

& $=$ & $2\beta_{2}e_{6}+\sum\limits_{i=3}^{n-4}\frac{(i+3)(i^{2}+3i+14)}{ i(i+1)(i+5)}\alpha_{i-1}e_{i+4}\ = \ 2\beta_{2}e_{6}+\sum\limits_{i=7}^{n}\frac{(i-1)(i^{2}-5i+18)}{ (i-4)(i-3)(i+1)}\alpha_{i-5}e_{i}.$
\end{longtable}
\noindent Comparing the presented expressions, we have  
\begin{longtable}{ll}
$\alpha_{k}=0,$ & $3\le k\le n-5;$ \\
$\gamma_{k}=\frac{(k+1)(k+2)^{2}}{k(k-1)(k+6)}\alpha_{k-2},$ & $4\le k\le n-3.$ \\
$\gamma_{1}=0;$ \ $\gamma_{2}=0;$ \ $\gamma_{3}=-\alpha_{1}+2\beta_{2};$ \\
\end{longtable}
\noindent The last gives 
\begin{longtable}{lcl}
$D(e_{1})$ & $=$ & $\alpha_{1}e_{1}+\alpha_{2}e_{2}+\sum\limits_{j=n-4}^n\alpha_{j}e_{j};$\\

$D(e_{2})$ & $=$ & $\beta_{2}e_{2}+\frac{5}{4}\alpha_{2}e_{3}+\frac{n(n-1)}{(n+2)(n-3)}\alpha_{n-4}e_{n-3}+\sum\limits_{j=n-2}^n\beta_{j}e_{j};$\\ 

$D(e_{3})$ & $=$ & $\left(-\alpha_{1}+2\beta_{2}\right)e_{3}+\frac{3}{2}\alpha_{2}e_{4}+\sum\limits_{j=n-2}^n\gamma_{j}e_{j};$\\
 
$D(e_{4})$ & $=$ & $2\alpha_{1}e_{4}+\frac{3}{2}\alpha_{2}e_{5}+\frac{n^{2}-5n+12}{(n-2)(n-3)}\alpha_{n-4}e_{n-1}+\frac{n^{2}-3n+8}{(n-1)(n-2)}\alpha_{n-3}e_{n};$\\

$D(e_{5})$ & $=$ & $\left(\alpha_{1}+\beta_{2}\right)e_{5}+\frac{7}{4}\alpha_{2}e_{6}+\frac{(n+1)(n^{2}-4n+12)}{(n+2)(n-2)(n-3)}\alpha_{n-4}e_{n}; $\\

$D(e_{6})$ & $=$ & $2\beta_{2}e_{6}+2\alpha_{2}e_{7};$\\

$D(e_{7})$ & $=$ & $30 D(e_{1}\diamond e_{4}) \ = \ 30 \left(\alpha_{1}e_{1}+\alpha_{2}e_{2}+\sum\limits_{j=n-4}^n\alpha_{j}e_{j} \right)\diamond e_{4}+$\\

\multicolumn{3}{r}{$+30 e_{1}\diamond\left(2\alpha_{1}e_{4}+\frac{3}{2}\alpha_{2}e_{5}+\frac{n^{2}-5n+12}{(n-2)(n-3)}\alpha_{n-4}e_{n-1}+\frac{n^{2}-3n+8}{(n-1)(n-2)}\alpha_{n-3}e_{n}\right) \ =$} \\

& $=$ & $\alpha_{1}e_{7}+ \alpha_{2}e_{8}+2\alpha_{1}e_{7}+\frac{15}{14}\alpha_{2}e_{8} \ = \ 3\alpha_{1}e_{7}+\frac{29}{14}\alpha_{2}e_{8};$ \\

$D(e_{7})$ & $=$ & $ 6D(e_{4}\diamond e_{1}) \ = \ 6 \big(D(e_{4})\diamond e_{1}+e_{4}\diamond D(e_{1}) \big)\ =$\\

& $=$ & $6\left(2\alpha_{1}e_{4}+\frac{3}{2}\alpha_{2}e_{5}+\frac{n^{2}-5n+12}{(n-2)(n-3)}\alpha_{n-4}e_{n-1}+\frac{n^{2}-3n+8}{(n-1)(n-2)}\alpha_{n-3}e_{n}\right)\diamond e_{1}+$ \\

\multicolumn{3}{r}{$+6e_{4}\diamond\left(\alpha_{1}e_{1}+\alpha_{2}e_{2}+\sum\limits_{j=n-4}^n\alpha_{j}e_{j} \right)\ =$} \\

& $=$ & $2\alpha_{1}e_{7}+\frac{3}{2}\alpha_{2}e_{8}+ \alpha_{1}e_{7}+\frac{1}{2}\alpha_{2}e_{8} \ = \ 3\alpha_{1}e_{7}+2\alpha_{2}e_{8};$ \\

$D(e_{7})$ & $=$ & $20D(e_{2}\diamond e_{3}) \ = \ 
20\big(D(e_{2})\diamond e_{3}+e_{2}\diamond D(e_{3}) \big) \ =$ \\

& $=$ & $20\left(\beta_{2}e_{2}+\frac{5}{4}\alpha_{2}e_{3}+\frac{n(n-1)}{(n+2)(n-3)}\alpha_{n-4}e_{n-3}+\sum\limits_{j=n-2}^n\beta_{j}e_{j} \right)\diamond e_{3}+$ \\

\multicolumn{3}{r}{$+20e_{2}\diamond\left(\left(-\alpha_{1}+2\beta_{2}\right)e_{3}+\frac{3}{2}\alpha_{2}e_{4}+\sum\limits_{j=n-2}^n\gamma_{j}e_{j} \right)\ =$} \\

& $=$ & $\beta_{2}e_{7}+\frac{5}{4}\alpha_{2}e_{8}+(2\beta_{2}-\alpha_{1}) e_{7}+ \alpha_{2}e_{8} \ = \ 
(3\beta_{2}-\alpha_{1}) e_{7}+\frac{9}{4}\alpha_{2}e_{8}.$
\end{longtable}
\noindent  That gives $\alpha_{2}=0$ and $\beta_{2}=\frac{4}{3}\alpha_{1},$ i.e.,  
\begin{longtable}{lcl}
$D(e_{1})$ & $=$ & $\alpha_{1}e_{1}+ \sum\limits_{j=n-4}^n\alpha_{j}e_{j};$ \\

$D(e_{2})$ & $=$ & $\frac{4}{3}\alpha_{1}e_{2}+\frac{n(n-1)}{(n+2)(n-3)}\alpha_{n-4}e_{n-3}+\sum\limits_{j=n-2}^n\beta_{j}e_{j};$ \\

$D(e_{3})$ & $=$ & $\frac{5}{3}\alpha_{1}e_{3}+\sum\limits_{j=n-2}^n\gamma_{j}e_{j};$ \\

$D(e_{4})$ & $=$ & $2\alpha_{1}e_{4}+\frac{n^{2}-5n+12}{(n-2)(n-3)}\alpha_{n-4}e_{n-1}+\frac{n^{2}-3n+8}{(n-1)(n-2)}\alpha_{n-3}e_{n};$ \\

$D(e_{5})$ & $=$ & $\frac{7}{3}\alpha_{1}e_{5}+\frac{(n+1)(n^{2}-4n+12)}{(n+2)(n-2)(n-3)}\alpha_{n-4}e_{n};$ \\

$D(e_{6})$ & $=$ & $\frac{8}{3}\alpha_{1}e_{6};$ \\

$D(e_{7})$ & $=$ & $3\alpha_{1}e_{7}.$ \\
\end{longtable}

Furthermore, using the induction method, we can prove that $D(e_{k})=\frac{k+2}{3}\alpha_{1}e_{k},$ for $k\ge 8.$ To do this, let us consider 
\begin{longtable}{lcl}
$D(e_{k+1})$ & $=$ & $ k(k-1)D(e_{1}\diamond e_{k-2}) \ = \ k(k-1) \big(D(e_{1})\diamond e_{k-2}+e_ {1}\diamond D(e_{k-2}) \big) \ = $ \\

& $=$ & $k(k-1) \left(\left(\alpha_{1}e_{1}+\sum\limits_{j=n-4}^n\alpha_{j}e_{j}\right)\diamond e_{k-2}+e_{1}\diamond\left(\frac{k}{3}\alpha_{1}e_{k-2}\right) \right) \ = \ \frac{k+3}{3}\alpha_{1}e_{k+1}.$ \\
\end{longtable}
\noindent  Joining all the obtained results, we see that $D$ is a linear combination of mappings
$\left\{ \varphi, \phi_i, \psi_j, \pi_k\right\}_{ 0\le i \le 4; \ 0\le j \le 2;\   0\le k \le 2}.$
\end{proof}

\begin{corollary}
$\mathfrak{Der}\big({\rm IDD}(K^1_\infty,0,-2)\big) = \big\langle \varphi\big\rangle,$ where $\varphi(e_{i}) =  \left(i+2\right) e_{i}.$
\end{corollary}

\end{document}